\documentclass[reqno, noend, a4paper,oneside]{amsart}

\date{7 November 2017}

\usepackage{amsthm,amsmath,amsfonts,mathrsfs,amssymb,todonotes,verbatim}
\usepackage{algorithm}
\usepackage{algorithmicx, algpseudocode}
\usepackage[T1]{fontenc}   
\usepackage{multirow}
\usepackage{url}
\usepackage{multirow}
\usepackage{rotating}
\usepackage{enumerate}
\usepackage{natbib}
\usepackage{epstopdf}
\usepackage{graphicx}
\usepackage{enumitem}
\usepackage{arydshln}
\usepackage[hang]{caption}
\usepackage[hyperfootnotes=false]{hyperref}

\algdef{SE}[DOWHILE]{Do}{doWhile}{\algorithmicdo}[1]{\algorithmicwhile\ #1}%

\theoremstyle{plain}
\newtheorem{theorem}{Theorem}[section]
\theoremstyle{definition}
\newtheorem{df}[theorem]{Definition}
\newtheorem{remark}[theorem]{Remark}
\newtheorem{notation}[theorem]{Notation}

\theoremstyle{plain}
\newtheorem{prp}[theorem]{Proposition}
\newtheorem{lemma}[theorem]{Lemma}

\newcommand{\Singular}{\textsc{Singular}}

\newcommand{\tY}{\widetilde{Y}}

\DeclareMathOperator{\Mon}{Mon}

\DeclareMathOperator{\id}{id}
\DeclareMathOperator{\parm}{par}
\DeclareMathOperator{\m}{\mathfrak{m}}
\DeclareMathOperator{\jet}{jet}

\DeclareMathOperator{\supp}{supp}
\DeclareMathOperator{\Jac}{Jac}
\DeclareMathOperator{\sign}{sign}
\DeclareMathOperator{\dash}{\!\textnormal{-}\!}

\DeclareMathOperator{\NF}{NF}

\DeclareMathOperator{\N}{\mathbb{N}}
\DeclareMathOperator{\Q}{\mathbb{Q}}
\DeclareMathOperator{\R}{\mathbb{R}}
\DeclareMathOperator{\C}{\mathbb{C}}
\DeclareMathOperator{\K}{\mathbb{K}}

\DeclareMathOperator{\coeff}{coeff}

\usepackage{enumitem}
\setlist[enumerate]{itemindent=*,leftmargin=*}

\begin{document}

\title[The Classification of Real Singularities Using \textsc{Singular}, %
Part III]{The Classification of Real Singularities Using \textsc{Singular} Part III: Unimodal Singularities of Corank 2}

\thanks{This research was supported by grant KIC14081491583 and grant 109327 (Incentive Funding for Rated Researches) of the National Research Foundation (NRF), by grants awarded by the University of Pretoria, and by SPP 1489 and SFB-TRR 195 of the German Research Foundation (DFG). Part of the research was done during a visit of the second author at the Riemann Center for Geometry and Physics at Leibniz Universit\"{a}t Hannover supported by a Riemann fellowship.}

\author{Janko B\"ohm}
\address{Department of Mathematics\\
University of Kaiserslautern\\
Erwin-Schr\"odinger-Str.\\
67663 Kaiserslautern\\
Germany}
\email{boehm@mathematik.uni-kl.de}

\author{Magdaleen S.\@ Marais}
\address{University of Pretoria and African Institute for Mathematical Sciences\\
Department of Mathematics and Applied Mathematics\\
Private bag X20\\
Hatfield 0028\\
South Africa}
\email{magdaleen.marais@up.ac.za}

\author{Andreas Steenpa\ss}
\address{Department of Mathematics\\
University of Kaiserslautern\\
Erwin-Schr\"odin\-ger-Str.\\
67663 Kaiserslautern\\
Germany}
\email{steenpass@mathematik.uni-kl.de}

\begin{abstract}

We present a classification algorithm for isolated hypersurface singularities of corank $2$ and modality $1$ over the real numbers. For a singularity given by a polynomial over the rationals, the algorithm determines its right equivalence class by specifying all representatives in Arnold's list of normal forms \citep{AVG1985} belonging to this class, and the corresponding values of the moduli parameter. 

We discuss how to computationally realize the individual steps of the algorithm for all singularities under consideration, and give explicit examples. The algorithm is implemented in the \textsc{Singular} library \texttt{realclassify.lib}.

\end{abstract}


\maketitle

\setlength{\topmargin}{0.5cm}
\setlength{\textheight}{21.8cm}
\setlength{\textwidth}{14cm}
\setlength{\oddsidemargin}{1cm}
\setlength{\evensidemargin}{1cm}

\section{Introduction}
 In the ground breaking work on singularities by \citet{A1974}, all isolated hypersurface singularities over the complex numbers up to modality $2$ have been classified. Arnold has also given fundamental theorems on the classification of real singularities up to modality $1$, which has been made explicit in \citet{AVG1985}. For his classification, Arnold considers stable equivalence of functions. Two function germs are stably equivalent if they are right equivalent after the  direct addition of a non-degenerate quadratic form. Over a field $\mathbb{K}$, two function germs $f,g\in\m^2\subset\K[[x_1,\ldots,x_n]]$, where $\m=\langle x_1,\ldots,x_n\rangle$, are right equivalent if there is a $\K$-algebra automorphism $\phi$ of $\K[[x_1,\ldots,x_n]]$ such that $\phi(f)=g$. 
 For $\mathbb{K}=\mathbb{C}$, modality $\leq 2$, and $\mathbb{K}=\mathbb{R}$, modality $\leq 1$, Arnold presents a finite list of normal forms, each of which is a family of polynomial equations with moduli parameters such that each equivalence class contains at least one, but finitely many, elements of these families. We refer to such elements as normal form equations.
 Over the complex numbers, Arnold gives an algorithmic determinator which, for certain classes of power series, determines a normal form, in which a representative of a given power series lies. In \citet{classify2} an algorithm is presented which, in addition, computes an explicit complex normal form equation for power series of modality and corank less or equal to $2$. For polynomial input, this algorithm has been implemented in the library \citep{classify2lib} for the computer algebra system \textsc{Singular} \citep{DGPS}. 

Considering singularities over the reals, an algorithm for computing the degenerate part and the inertia index is given in \citet{realclassify1}. For singularities of modality $0$, an algorithm to determine the corresponding normal form is developed in the same article. 

In the case of singularities of modality $1$, it turns out that an equivalence class may contain several normal form equations, as specified by Arnold. In \citet{realclassify2}, a complete classification of singularities up to modality $1$ and corank $2$ is given, in the sense that all complex and real equivalences between complex, respectively real, normal form equations are determined.

In this paper, we develop a determinator for real singularities of modality $1$ and corank $2$, which computes, for a given input polynomial, all normal form equations in Arnold's list (see Table \ref{tab:normal_forms}). In this setting, the complex types correspond to real main types, which split up into real subtypes by modifying the signs of the terms in the normal form of the real main type, except in the case $Y_{r,r}$.  This complex normal form splits up into the real main types $\widetilde Y_r$ and $Y_{r,r}$. 

In fact, we describe an algorithm which computes, for an arbitrary input polynomial $f\in\m^3\subset\Q[x,y]$, the following data:
all real singularity subtypes of $f$ as well as all normal form equations in the right equivalence class of $f$ with the respective parameter given as the unique real root of a minimal polynomial over $\Q$ in a specified interval. The algorithm is implemented in the \textsc{Singular} library \texttt{realclassify.lib}. 

This paper is structured as follows: In Section \ref{Definitions and Preliminary Results}, we introduce the fundamental definitions and give the required background on the classification of singularities. In Section \ref{section:generalAlg}, we develop a general algorithm for the classification of real singularities of modality $1$ and corank $2$. Although the general algorithm is applicable to all cases under consideration, some steps do not have a straight-forward implementation for  certain types of singularities. Therefore, in the subsequent sections, we give explicit computational realizations for all steps of the algorithm. The exceptional cases ($W_{12}, W_{13},Z_{11}, Z_{12}, Z_{13},  E_{12}, E_{13}, E_{14} $) follow the general algorithm in a direct way. In Section \ref{sec:Excep}, as an example, we give an explicit algorithm for the case $E_{14}$. Section \ref{sec:Parab} handles the parabolic cases ($X_9, J_{10}$). Section \ref{sec:Hyper} deals with the hyperbolic cases ($X_{9+k}$, $ J_{10+k}$, $Y_{r,s}$, $\widetilde Y_r$). The cases $X_{9+k}$ and $ J_{10+k}$ follow the general algorithm in a straight-forward manner, whereas the cases $Y_{r,s}$ and $\widetilde Y_r$ require some attention to detail.

\vspace{0.1in} 

\noindent\emph{Acknowledgements}. We would like to thank Gert-Martin Greuel, and Gerhard Pfister for many fruitful discussions, and Domnik Bendle and Clara Petroll for implementing a multivariate real root isolation algorithm in the \Singular-library  \texttt{rootisolation.lib} \cite{rootisolationlib}.

\section{Definitions and Preliminary Results}\label{Definitions and Preliminary Results}
In this section we give some basic definitions and results, as well as some notation that will be used throughout the paper.

We only consider classes of germs with respect to right equivalence.

\begin{df} 
Let $\K$ be a field. Two power series $f,g\in\K[[x_1,\ldots,x_n]]$ are called {\bf right equivalent}, denoted by $f\overset{\K}{\sim}g$, if there exists a $\K$-algebra automorphism $\phi$ of $\K[[x_1,\ldots,x_n]]$ such that $\phi(f)=g$. If $\K=\R$, we also write $\sim$ to denote $\overset{\R}\sim$.
\end{df}

Using this equivalence relation, \cite{A1974} gives the following formal definition of a normal form. From now on, let $\K$ be either $\R$ or $\C$. 

\begin{df}\label{def nfequ}
Let $K\subset \K[[x_1,\ldots, x_n]]$ be a union of equivalence classes with respect to the relation $\overset{\K}{\sim}$. A {\bf normal form} for $K$ is given by a smooth map
\[\Phi:B\longrightarrow \K[x_1,\ldots,x_n]\subset\K[[x_1,\ldots,x_n]]\]
of a finite dimensional $\K$-linear space $B$ into the space of polynomials for which the following three conditions hold:
\begin{itemize}[leftmargin=10mm]
\item[(1)] $\Phi(B)$ intersects all equivalence classes of $K$;
\item[(2)] the inverse image in $B$ of each equivalence class is finite;
\item[(3)] $\Phi^{-1}(\Phi(B)\setminus K)$ is contained in a proper hypersurface in $B$.
\end{itemize}
We call the elements of the image of $\Phi$ {\bf normal form equations}.
\end{df}

\begin{remark}
There is a type $T$ associated to each of Arnold's corank $2$ real normal forms of modality $0$ and $1$ (see Table \ref{tab:normal_forms}), and we denote the corresponding normal form by $\NF(T)$. Depending on whether $T$ is a real or a complex type, we consider $\K=\R$ or $\K=\C$. For $b\in \parm(\NF(T)):=\Phi^{-1}(K)$ with $K$ as in Definition \ref{def nfequ}, we write $\NF(T)(b)=\Phi(b)$ for the corresponding normal form equation.
\end{remark}

We briefly introduce the concepts of weighted jets, filtrations, Newton polygons, and permissible chains. For
background regarding the definitions in this section, we refer to
\citet{A1974} and \cite{PdJ2000}.

\begin{df}
Let $w=(c_1,\ldots,c_n)\in\N^n$ be a weight on the variables $(x_1,\ldots,x_n)$. We define the weighted degree on $\Mon(x_1,\ldots,x_n)$ by $w\dash\deg(\prod_{i=1}^nx_i^{s_i})=\sum_{i=1}^n c_i s_i$. If the weight on all variables is $1$, we call the weighted degree of a monomial $m$ the standard degree of $m$ and write $\deg(m)$ instead of $w\dash\deg(m)$.

A polynomial is called  {\bf quasihomogeneous} of degree $d$ with respect to $w$ if $w\dash\deg(t)=d$ for any term $t$ of $f$.

\end{df}

\begin{df}\label{def:piecewiseWeight}
Let $w=(w_1,\ldots,w_s)\in(\N^n)^s$ be a finite family of weights on the variables $(x_1,\ldots,x_n)$. For any monomial $m\in \K[x_1,\ldots,x_n]$ (or term $t=c\cdot m$, $c\in \K$),  we define the {\bf piecewise weight} with respect to $w$ as
$$
w\dash\deg(m):=\min_{i=1,\ldots,s}w_i\dash\deg(m).
$$
\end{df}

\begin{df}
Let $w$ be a (piecewise) weight on $\Mon(x_1,\ldots,x_n)$.
\begin{enumerate}[leftmargin=10mm]
\item
Let $f = \sum_{i = 0}^{\infty} f_{i}$ be the decomposition of
$f \in \K[[x_1,\ldots,x_n]]$ into weighted homogeneous parts $f_{i}$ of
$w$-degree $i$. We denote the {\bf weighted $j$-jet} of $f$ by
\[
w \dash \jet(f, j) := \sum_{i = 0}^j f_{i} \,.
\]
The sum of terms of $f$ of lowest $w$-degree is the {\bf principal part} of $f$ with respect to~$w$.

\item A power series in $\K[[x_1,\ldots,x_n]]$ has {\bf filtration} $d \in \N$ if all its
monomials are of weighted degree $d$ or higher. The power series of filtration
$d$ form a vector space
\[
E_d^w \subset \K[[x_1,\ldots,x_n]] \,.
\]
\item A power series $f\in\K[[x_1,\ldots,x_n]]$ is a {\bf germ with non-degenerate Newton principal part} if $f$ has filtration $d\in\N$ with respect to $w$ and if
the  principal part $w\dash\jet(f,d)$ is non-degenerate, that is, if its Milnor number is finite. If $w$ consists out of a single weight, we call $f$ {\bf semi-quasihomogeneous} and $w\dash\jet(f,d)$ the {\bf quasihomogeneous part} of $f$.
\item A power series $f\in\K[[x_1,\ldots,x_n]]$ is {\bf weighted $k$-determined} with respect to the weight $w$ if
\[ f\sim w\dash\jet(f,k)+g\qquad\text{for all } g\in E_{k+1}^w.\]
We define the {\bf weighted determinacy} of $f$ as the minimum number $k$ such that $f$ is $k$-determined.
\end{enumerate}
\end{df}

\begin{remark}
\begin{enumerate}[leftmargin=10mm]
\item If for a given type $T$, $w\dash\jet(\NF(T)(b),j)$ is independent of $b\in\parm(\NF(T))$, we denote it by $w\dash\jet(T,j)$. 
\item Note that $d < d'$ implies $E_{d'}^w \subseteq E_d^w$. Since elements of
the product $E_{d'}^w \cdot E_d^w$ have filtration $d'+d$, it follows that $E_{d}^w$ is an
ideal in the ring of power series. 
\item If the weight
of each variable is $1$, we write $E_d$ and $\jet(f,j)$ instead of  $E_d^w$ and $w\dash\jet(f,j)$, respectively. 
\end{enumerate}
\end{remark}

There are also similar concepts for coordinate transformations:

\begin{df}\label{phi}
Let $\phi$ be a $\K$-algebra automorphism of $\K[[x_1,\ldots,x_n]]$ and let
$w$ be a weight on $\Mon(x_1,\ldots,x_n)$.

\begin{enumerate}[leftmargin=10mm]
\item
For $j > 0$ we define the \emph{$w\dash\jet(\phi,j)$}, denoted by $\phi_j^w$,
to be the automorphism given by
\[
\phi_j^w(x_i) := w\dash\jet(\phi(x_i),w\dash\deg(x_i)+j) \quad
\text{for all }i = 1,\ldots,n \,.
\]
If the weight of each variable is $1$, that is, \@ $w = (1, \ldots, 1)$, we simply
write $\phi_j$ for $\phi_j^w$.

\item\label{enum:filtration}
$\phi$ has filtration $d$ if, for all $\lambda \in \N$,
\[
(\phi-\id)E_\lambda^w \subset E_{\lambda+d}^w \,.
\]
\end{enumerate}
\end{df}

\begin{remark}
Note that $\phi_0(x_i) = \jet(\phi(x_i), 1)$ for all $i = 1, \ldots, n$.
Furthermore note that $\phi_0^w$ has filtration $\le 0$ and
that $\phi_j^w$ has filtration $j$ if $j > 0$ and $\phi_{j-1}^w = \id$.
\end{remark}

For the next definition we restrict ourselves to the power series ring $\K[[x,y]]$ in two variables.
 
 \begin{df}
 Let $f=\sum_{i,j}a_{i,j}x^{i}y^{j}\in\K[[x,y]]$, let $T$ be a corank $2$ singularity type and $0\neq b\in\parm(\NF(T))$. We call
 \begin{eqnarray*}
 \supp(f)&:=&\{x^{i}y^{j}\ |\ a_{i,j}\neq 0\} \text{ and }\\
 \supp(T)&:=&\supp(\NF(T)(b)),\\
 \end{eqnarray*}
 where $b\in\operatorname{par}(\NF(T))$ is generic, the support of $f$ and $\NF(T)$, respectively. We write $\supp(T,j):=\supp(\jet(\NF(T)(b),j))$ for $j\geq 0$. Let
$$\Gamma_+(T):=\displaystyle{\bigcup_{x^{i} y^{j}\in\supp(T)}}((i,j)+\R^2_+)$$
and let $\Gamma(T)$ be the boundary of the convex hull of $\Gamma_+(T)$ in $\R^2$. Then:
\begin{enumerate}[leftmargin=10mm]
\item $\Gamma(T)$ is called the {\bf Newton polygon} of $\NF(T)$.
\item The compact segments of $\Gamma(T)$ are called {\bf faces}. 
\item Let $\Delta$ be a face of $\Gamma(T)$. Then $\Delta$ induces a weight $w$ on $\Mon(x,y)$ in the following way: If $\Delta$ has slope $-\frac{w_x}{w_y}$, in lowest terms, and $w_x,w_y>0$, we set $w\dash\deg(x)=w_x$ and $w\dash\deg(y)=w_y$. 
 \item If $w_1,\ldots,w_s$ are the weights associated to the faces of $\Gamma (T)$, ordered by increasing slope, there are unique minimal integers $\lambda_1,\ldots,\lambda_s\geq 1$ such that the piecewise weight associated to $w(T)=(\lambda_1 w_1,\ldots,\lambda_s w_s)$ by Definition \ref{def:piecewiseWeight} is constant on $\Gamma(T)$.  We denote this constant by $d(T)$. 
\item A monomial $m$ lies strictly underneath, on or above $\Gamma(T)$ if, respectively, the $w(T)$-degree of $m$ is less than, equal to or greater than $d(T)$.
 \end{enumerate}
 \end{df}

 \begin{notation}
Given  $f\in\mathbb{K}[[x,y]]$ and $m\in \Mon(x,y)$, we write $\coeff(f,m)$ for the coefficient of $m$ in $f$. 
\end{notation}

 \begin{df}\label{def:segments}
 Let $f\in\K[x,y]$ be piecewise quasihomogeneous such that each partial derivative of $f$ is a binomial. 
 \begin{enumerate}[leftmargin=10mm]
 \item The convex hull of the exponent vectors of the monomials of $\frac{\partial f}{\partial x}$ is called the {\bf fundamental $x$-segment}, similarly for $\frac{\partial f}{\partial y}$.
 \item A translate  of a fundamental segment by an integral vector with non-negative components is called a {\bf permissible segment}.
 \item Two permissible segments are said to be {\bf joined} if they have a common end.
 \item A {\bf permissible chain} is a sequence of permissible segments such that each consecutive two of them are joined.
 \end{enumerate}
 \end{df}
 
 See \cite[Example 9.6]{A1974} for an example.
 
 \begin{remark}
 \begin{enumerate}[leftmargin=10mm]
 \item The assumptions of Definition \ref{def:segments} are satisfied for the principal parts of all corank $2$ singularities of modality $0$ and $1$.
 \item Suppose $f$ is a corank $2$ singularity of modality $1$ whose Newton polygon has two faces. Then a permissible $x$-segment is not parallel to any permissible $y$-segment, that is, they are either disjoint, or intersect in a single point.
 \end{enumerate}
 \end{remark}
 
 \begin{df}
 The {\bf Jacobian ideal} $\Jac(f)\subset \K[[x,y]]$ of $f$ is generated by the partial derivatives of $f$. The {\bf local algebra} $Q(f)$ is the residue class ring of the Jacobian ideal of~$f$.
 \end{df}
  
 \begin{prp}[{\citet{A1974}, Proposition 9.7}]
 If the exponent vector of a monomial lies in an infinite permissible chain, then it lies in the Jacobian ideal of $f$.
 \end{prp}

 In \citep[Section 7]{A1974} the following results are used for the classification of singularities of corank~$2$. 

\begin{df}
Suppose $f$ is a germ, $e_1,\ldots,e_\mu$ are monomials representing a basis of the local algebra of $f$, and $e_1,\ldots, e_s$ are the monomials in this basis above or on $\Gamma(T)$. We then call $e_1,\ldots, e_s$ a {\bf system} of the local algebra of $f$. 
\end{df}

\begin{theorem}
Let $f$ be a semi-quasihomogeneous function with quasihomogeneous part $f_0$, and let $e_1,\ldots,e_\mu$ be monomials representing a basis of the local algebra of $f_0$. Then $e_1,\ldots,e_\mu$ also represent a basis of the local algebra of $f$. 
\end{theorem}
 
 \begin{theorem}\label{thm:quasiNF}
 Let $f$ be a semi-quasihomogeneous function with quasihomogeneous part $f_0$ and let $e_1,\ldots,e_s$ be a system of the local algebra of $f_0$. Then $f$ is equivalent to a function of the form $f_0+\sum c_ke_k$.
 \end{theorem}
Arnold proved Theorem \ref{thm:quasiNF} by iteratively applying the following lemma.

\begin{lemma}\label{lemma:quasiNF}
Let $f_0$ be a quasihomogeneous function of weighted $w$-degree $d_w$ and let $e_1,\ldots,e_r$ be the monomials of a given degree $d'>d_w$ in a system of the local algebra of $f_0$. Then, for every series of the form $f_0+f_1$, where the filtration of $f_1$ is greater than $d_w$, we have
\[f_0+f_1\sim f_0+f'_1,\]
where the terms in $f'_1$ of degree less than $d'$ are the same as in $f_1$, and the part of degree $d'$ can be written as $c_1e_1+\cdots+c_re_r$, $c_1,\ldots,c_r\in\R$.
\end{lemma}

\begin{proof}
Let $g({\bf x})$ denote the sum of the terms of degree $d'$ in $f_1$. There exists a decomposition of $g$ of the form
\[g({\bf x})=\sum_i\frac{\partial f_0}{\partial x}v_i({\bf x})+c_1e_1+\cdots +c_re_r,\]
since $e_1,\ldots,e_r$ represents a monomial vectorspace basis for the local algebra of $f_0$. Applying the transformation defined by
\[x_i\mapsto x_i-v_i({\bf x})\]
to $f$, we transform $f$ to
\[f_0({\bf x})+\left( f_1({\bf x})+(c_1e_1({\bf x})+\cdots+c_re_r({\bf x})-g({\bf x})\right)+R({\bf x}),\]
where the filtration of $R$ is greater than $d'$.
\end{proof}

The following result is proved for $a\ge 4$, $b\ge 5$ and $\K=\C$ by
\citet{A1974}. The same proof is also valid for $a=3$, $b\ge 7$. Moreover the proof is also valid for $\K=\R$. Hence we obtain the following more general result. 

\begin{lemma}\label{principalpart}
Let $\K$ be either $\R$ or $\C$. Suppose that $f\in\K[[x,y]]$ is a germ with non-degenerate Newton principal part $f_0=x^a+\lambda x^2y^2+y^b$, where
$0\neq\lambda\in\K$ and $a,b\in\N$, such that $\Gamma(f_0)$ has two faces. Then
\[f\sim f_0.\]
\end{lemma}

\begin{remark}
Note that it follows from Lemma \ref{lemma:quasiNF} and Lemma \ref{principalpart} that all germs with a non-degenerate Newton principal part of modality $1$, corank $2$ are finitely weighted determined. Moreover, we explicitly obtain the weighted determinacy for every such germ.
\end{remark}

 \begin{notation}
 A system of a local algebra is in general not unique. For his lists of normal forms of hypersurface singularities with non-degenerate Newton boundary, Arnold has chosen a specific system for the local algebra in each case. In the rest of the paper, we call these systems the \textbf{Arnold systems}.
 
 We refer to a set of parameters monomials in a normal form as a \textbf{parameter system}. Again for his lists of normal forms, Arnold has chosen in each case a specific parameter system. We refer to these as the \textbf{Arnold parameter systems}.
 
 Note that in the case of a normal form with  non-degenerate Newton boundary, corank $2$ and modality $1$, a system is a parameter system.
 \end{notation}
 We proof the following result in addition to \cite[Lemma 9]{realclassify1} which has similar results for linear factorizations.
 
 \begin{lemma}\label{J10+kfactorization}
Let $n\in\N$, $n\ge 2$. Suppose $f\in\Q[x,y]$ is homogeneous and factorizes as $g_1^n(g_2)$ over $\R$, where $g_1$ is a polynomial of degree $1$ and $g_2$ is a polynomial of degree $2$ that does not have a multiple root over $\C$. Then $f=ag_1'^ng_2'$, where $g_1'$ is a polynomial of degree $1$ over $\Q$, $g_2$ is a polynomial of degree 2 over $\Q$ and $a\in\Q$.
\end{lemma}

\begin{proof}
Let $f=(a_1x+a_2y)^n(a_3x^2+a_4xy+a_5y^2)$ with $a_i\in\R$. Since the coefficient of $x^{n+2}$ in $f$ is a rational number, it follows that 
\[(x+\frac{a_2}{a_1}y)^n(x^2+\frac{a_4}{a_3}xy+\frac{a_5}{a_3}y^2)\in\Q[x,y].\]
Since $\Q$ is a perfect field, $(x+\frac{a_2}{a_1}y)(x^2+\frac{a_4}{a_3}xy+\frac{a_5}{a_3}y^2)\in\Q[x,y]$. Hence
$(x+\frac{a_2}{a_1}y)^{n-1}\in\Q[x,y]$ which implies that $x+\frac{a_2}{a_1}y\in\Q[x,y]$. Therefore $f=ag_1'^ng_2'$, where
$g_1'=x+\frac{a_2}{a_1}y$, $g_2'=x^2+\frac{a_4}{a_3}xy+\frac{a_5}{a_3}y^2$ and $a=a_1^na_3$.
\end{proof}

\begin{table}[p]
\centering
\caption{Normal forms of singularities of modality~$1$ and corank~$2$ as given
in \citet{AVG1985}\medskip}
\label{tab:normal_forms}
\begin{minipage}{\textwidth}
\renewcommand{\thempfootnote}{\fnsymbol{mpfootnote}}
\addtocounter{mpfootnote}{1}
\newcommand{\setfnA}{\footnote{\label{fnA}%
Note that the restriction $a^2 \neq 4$ applies to the normal forms of the real
subtypes $X_9^{++}$, $X_9^{--}$, and $J_{10}^+$ as well as to the normal forms
of the complex types $X_9$ and $J_{10}$. 
}}
\newcommand{\reffnA}{\textsuperscript{\ref*{fnA}}}
\centering
\begin{tabular}{|c|c|c|c|c|}
\hline

\multicolumn{1}{|c}{}
 & & Complex     & Normal forms     & \multirow{2}{*}{Restrictions} \\[-0.5ex]
\multicolumn{1}{|c}{}
 & & normal form & of real subtypes &                               \\
\hline\hline

\multirow{6}{*}{\begin{sideways}Parabolic\end{sideways}}

& \multirow{4}{*}{$X_9$} & \multirow{4}{*}{$x^4+ax^2y^2+y^4$}
  & $+x^4+ax^2y^2+y^4$ $(X_9^{++})$ & \multirow{2}{*}{$a^2\neq+4$\setfnA}
\\ \cline{4-4}
&&& $-x^4+ax^2y^2-y^4$ $(X_9^{--})$ &
\\ \cline{4-5}
&&& $+x^4+ax^2y^2-y^4$ $(X_9^{+-})$ & 
\\ \cline{4-4}
&&& $-x^4+ax^2y^2+y^4$ $(X_9^{-+})$ &
\\ \cline{2-5}

& \multirow{2}{*}{$J_{10}$} & \multirow{2}{*}{$x^3+ax^2y^2+xy^4$}
  & $x^3+ax^2y^2+xy^4$ $(J_{10}^+)$ & $a^2 \neq +4$\reffnA \\ \cline{4-5}
&&& $x^3+ax^2y^2-xy^4$ $(J_{10}^-)$ &  
\\ \hline

\multirow{12}{*}{\begin{sideways}Hyperbolic\end{sideways}}

& \multirow{2}{*}{$J_{10+k}$} & \multirow{2}{*}{$x^3+x^2y^2+ay^{6+k}$}
  & $x^3+x^2y^2+ay^{6+k}$ $(J_{10+k}^+)$
      & \multirow{2}{*}{$a \neq 0,\; k > 0$} \\ \cline{4-4}
&&& $x^3-x^2y^2+ay^{6+k}$ $(J_{10+k}^-)$ &   \\ \cline{2-5}

& \multirow{4}{*}{$X_{9+k}$} & \multirow{4}{*}{$x^4+x^2y^2+ay^{4+k}$}
  & $+x^4+x^2y^2+ay^{4+k}$ $(X_{9+k}^{++})$
      & \multirow{4}{*}{$a \neq 0,\; k > 0$}  \\ \cline{4-4}
&&& $-x^4-x^2y^2+ay^{4+k}$ $(X_{9+k}^{--})$ & \\ \cline{4-4}
&&& $+x^4-x^2y^2+ay^{4+k}$ $(X_{9+k}^{+-})$ & \\ \cline{4-4}
&&& $-x^4+x^2y^2+ay^{4+k}$ $(X_{9+k}^{-+})$ & \\ \cline{2-5}

& \multirow{4}{*}{$Y_{r,s}$} & \multirow{4}{*}{$x^2y^2+x^r+ay^s$}
  & $+x^2y^2+x^r+ay^s$ $(Y_{r,s}^{++})$
      & \multirow{4}{*}{$a \neq 0,\; r,s > 4$} \\ \cline{4-4}
&&& $-x^2y^2-x^r+ay^s$ $(Y_{r,s}^{--})$ &      \\ \cline{4-4}
&&& $+x^2y^2-x^r+ay^s$ $(Y_{r,s}^{+-})$ &      \\ \cline{4-4}
&&& $-x^2y^2+x^r+ay^s$ $(Y_{r,s}^{-+})$ &      \\ \cdashline{2-3}\cline{4-5}

& \multirow{2}{*}{$\tY_r$} & \multirow{2}{*}{$(x^2+y^2)^2+ax^r$}
  & $+(x^2+y^2)^2+ax^r$ $(\tY_r^+)$
      & \multirow{2}{*}{$a \neq 0,\; r > 4$} \\ \cline{4-4}
&&& $-(x^2+y^2)^2+ax^r$ $(\tY_r^-)$ &        \\ \hline

\multirow{12}{*}{\begin{sideways}Exceptional\end{sideways}}

& $E_{12}$ & $x^3+y^7+axy^5$ & $x^3+y^7+axy^5$ & - \\ \cline{2-5}

& $E_{13}$ & $x^3+xy^5+ay^8$ & $x^3+xy^5+ay^8$ & - \\ \cline{2-5}

& \multirow{2}{*}{$E_{14}$} & \multirow{2}{*}{$x^3+y^8+axy^6$}
  & $x^3+y^8+axy^6$ $(E_{14}^+)$ & \multirow{2}{*}{-} \\ \cline{4-4}
&&& $x^3-y^8+axy^6$ $(E_{14}^-)$ &                    \\ \cline{2-5}

& $Z_{11}$ & $x^3y+y^5+axy^4$ & $x^3y+y^5+axy^4$ & - \\ \cline{2-5}

& $Z_{12}$ & $x^3y+xy^4+ax^2y^3$ & $x^3y+xy^4+ax^2y^3$ & - \\ \cline{2-5}

& \multirow{2}{*}{$Z_{13}$} & \multirow{2}{*}{$x^3y+y^6+axy^5$}
  & $x^3y+y^6+axy^5$ $(Z_{13}^+)$ & \multirow{2}{*}{-} \\ \cline{4-4}
&&& $x^3y-y^6+axy^5$ $(Z_{13}^-)$ &                    \\ \cline{2-5}

& \multirow{2}{*}{$W_{12}$} & \multirow{2}{*}{$x^4+y^5+ax^2y^3$}
  & $+x^4+y^5+ax^2y^3$ $(W_{12}^+)$ & \multirow{2}{*}{-} \\ \cline{4-4}
&&& $-x^4+y^5+ax^2y^3$ $(W_{12}^-)$ &                    \\ \cline{2-5}

& \multirow{2}{*}{$W_{13}$} & \multirow{2}{*}{$x^4+xy^4+ay^6$}
  & $+x^4+xy^4+ay^6$ $(W_{13}^+)$ & \multirow{2}{*}{-} \\ \cline{4-4}
&&& $-x^4+xy^4+ay^6$ $(W_{13}^-)$ &                    \\ \hline

\end{tabular}
\medskip
\end{minipage}
\end{table}

\section{General Classification Algorithm}\label{section:generalAlg}
In this section we outline the general structure of an algorithm (see Algorithm \ref{alg:ClassPara}) to determine, for a given input polynomial $f\in\Q[x,y]$, the real types as well as, for each type, the corresponding normal form equations to which $f$ is equivalent (see Table \ref{tab:normal_forms}). Each parameter is given as the unique root of its minimal polynomial over $\Q$ in a specified interval. 

Figures \ref{fig exceptional W} to \ref{fig parabolic} show in the gray shaded area all monomials which can possibly occur in a polynomial $f$ of the given type $T$. The Newton polygon $\Gamma(T)$ is shown as the blue line with the non-moduli monomials of $\NF(T)$ shown as blue dots. Red dots indicate monomials which are not in $\Jac(f)$. The dot with the thick black circle indicates the moduli monomial in the Arnold system.

We start the classification process by classifying the input polynomial according to the complex classification in \cite{AVG1985}. We use the \textsc{Singular} library {\tt classify2.lib} \citep{classify2lib} for this purpose. As shown by \cite{A1974}, the complex types correspond to real main types, which split up into real subtypes by modifying the signs of the terms in the normal form of the real main type, except in the case $Y_{r,r}$. This complex type splits up into the real main types $\widetilde Y_r$ and $Y_{r,r}$. The $\widetilde Y_r$ case differs from all the other cases in the sense that these singularities are not right equivalent over the real numbers to a germ with a non-degenerate Newton principal part.\footnote{Thus it does not make sense to include the $\widetilde Y_r$ case in the Figures \ref{fig exceptional W} to \ref{fig parabolic}.} Therefore we postpone the discussion of the case $\widetilde Y_r$ to Section \ref{sec Yr}, where we will give a modification of the general algorithm applicable to this case.

\begin{remark}
We can easily distinguish whether an input polynomial $f$ is of real main type $Y_{r,r}$ or $\widetilde Y_r$, using \cite[Proposition 8]{realclassify1} by considering the number of real roots (i.e., real homogeneous linear factors) of $\jet(f,4)$. The $4$-jet of functions of real type $Y_{r,s}$ has four real roots (counted with multiplicity), while the $4$-jet of functions of real type $\widetilde Y_r$ does
not have any real roots. The \textsc{Singular} library {\tt rootsur.lib} \citep{roots} can be used to determine the number of roots. 

\end{remark}

 In all the other cases, the Newton polygon $\Gamma(T)$ of a real subtype $T$ coincides with the Newton polygon of its corresponding complex and main real type.  
 
\begin{remark}Note that, according to \cite{realclassify2}, normal form equations of different complex types, and hence of different real main types, cannot be equivalent. Normal form equations of types $Y_{r,r}$ and $\widetilde Y_r$ can be complex equivalent, but not real equivalent.
\end{remark}

Suppose $\NF(T)(b)$, $b\in\parm(\NF(T))$, is one of the normal form equations to which $f$ is equivalent. Then there exists an $\R$-algebra automorphism $\phi$ of $\R[[x,y]]$ such that $f=\phi(\NF(T)(b))$.
We write $w=w(T)$ for the piecewise weight induced by $\Gamma(T)$ and let $d_w=d(T)$ be the degree of the monomials on $\Gamma(T)$ with respect to $w$. 

Our first goal is to transform $f$, and thus $\phi$, iteratively by composing $\phi$ with a suitable $\R$-algebra automorphism such that, after every step, we have \[f=\phi(\NF(T)(b))\]and, after step $i$, we have \[\supp(\jet(f,i))=\supp(T,i)\] if $\Gamma(T)$ has one face, and \[\supp(w\dash\jet(\jet(f,i),0)=\supp(T,i)\] if $\Gamma(T)$ has two faces. This implies that, after a finite number of steps, we may assume, by an appropriate choice of $\phi$, that \[\phi_0^w(x)=c_xx\qquad\text{and}\qquad \phi_0^w(y)=c_yy,\] where $w$ is the weight defined by $\Gamma(T)$. After this, in the case where $\Gamma(T)$ has two faces, Theorem  \ref{principalpart} can be used to determine the equivalence class of $f$ by eliminating the monomials above $\Gamma(T)$. 
In the case where $\Gamma(T)$ has only one face, the algorithmic proof of Theorem \ref{thm:quasiNF} can be used to eliminate all monomials above $\Gamma(T)$ which are not in $\supp(T)$, by again iteratively transforming $f$ and $\phi$ such that we have
\[f=\phi(\NF(T)(b))\] after each step, and  \[\supp(w\dash\jet(f,d_w+j_i))=\supp(w\dash\jet(T,d_w+j_i)),\] where $1\leq j_1<\cdots<j_i$, after step $i$. We may thus assume that 
\[\phi_{j_i}^w(x)=c_xx\qquad\text{and}\qquad \phi_{j_i}^w(y)=c_yy,\] where $c_x,c_y\in\R^*$.
Since the determinacy $k$ of $f$ is finite, terms of degree greater than $k$ may be discarded after each step. 
Once $f$ is transformed to a germ with non-degenerate Newton principal part and weighted determinacy bound $k'$, the elimination process above $\Gamma(T)$ can be stopped once $d_w+j_i>k'$. 
Hence a normal form equation of $f$ can be determined in finitely many steps. To summarize, the general classification algorithm involves the following main steps:
\begin{enumerate}[leftmargin=10mm]
\item[I.] Determine  the complex type and the corresponding Newton polygon $\Gamma(T)$. 
\item[II.] Eliminate all monomials in $\supp(f)\setminus\supp(T)$ underneath or on $\Gamma(T)$.
\item[III.] Eliminate all monomials in $\supp(f)\setminus\supp(T)$ above $\Gamma(T)$.
\item[IV.] Read off the real subtypes and the corresponding normal form equations.
\end{enumerate}

\noindent Step (I) is straightforward using \texttt{classify2.lib}. We now discuss Steps (II) to (IV) in detail:

\subsection*{Step (II): Elimination of monomials underneath and on the Newton Polygon}

We first eliminate the monomials in $\jet(f,d)$, which are not in $\jet(T,d)$, where $d$ is the maximum filtration of $f$ with respect to the standard degree. This can be done by a linear transformation over $\R$, since these terms can only be created from the monomials in $\phi_0(x)$ and $\phi_0(y)$. Removing these terms amounts to transforming $\phi$ such that \begin{equation}\label{trans:xn}
\phi_0(x)=c_xx,\ c_x\in\R,
\end{equation}
if $\jet(T,d)=x^n$ for some $n\in\N$, or such that 
\begin{equation}\label{trans:yn}
\phi_0(x)=c_xx,\ c_x\in\R,\text{ and }\phi_0(y)=c_yy,\ c_y\in\R,
\end{equation}
if $\jet(T,d)\neq x^n$ for any $n\in\N$. For the exceptional cases and the parabolic cases $J_{10}$ and $J_{10+k}$, this can be done by factorizing $\jet(f,d)$ over $\Q$ using \cite[Proposition 8]{realclassify1} and \cite[Lemma 9]{realclassify1}. Similarly, can we use Lemma~\ref{J10+kfactorization} for the case $X_{9+k}$. See Algorithms \ref{alg:1jetEx} and~\ref{alg:1jetX9+k}, respectively, which implement the required transformations.

\begin{algorithm}[ht]
\caption{}%
\label{alg:1jetEx}
\begin{algorithmic}[1]
\Require{$f\in\Q[x,y]$ of real type $T$, where $T$ is exceptional, or parabolic of type $J_{10}$ or $J_{10+k}$, with maximum filtration $d$}
\Ensure{$h\in\Q[x,y]$ with $f\sim h$ and $\supp(\jet(h,d))=\supp(T,d)$}
\State $g:=\jet(f,d)$
\State factorize $g$ over $\Q$ as $cf_1^\alpha f_2^\beta$, with $f_1$ and $f_2$ linear and non-associated, $c\in\Q$, and $\alpha>\beta\geq 0$
\If{$\beta>0$}
\State apply $f_1\mapsto x$, $f_2\mapsto y$ to $f$
\ElsIf{$f_1\neq c'y$, $c'\in\Q$}
\State apply $f_1\mapsto x$, $y\mapsto y$ to $f$
\Else  
\State apply $f_1\mapsto x$, $x\mapsto y$ to $f$
\EndIf
\Return $f$
\end{algorithmic}
\end{algorithm}
 
\begin{algorithm}[ht]
\caption{}%
\label{alg:1jetX9+k}
\begin{algorithmic}[1]
\Require{$f\in\Q[x,y]$ of real type $T=X_{9+k}$}
\Ensure{$h\in\Q[x,y]$ with $f\sim h$  and $\supp(\jet(h,4))=\supp(T,4)$}
\State $g:=\jet(f,4)$
\State factorize $g$ over $\Q$ as $cf_1^2f_2$ with $f_1$ linear, $f_2$ quadratic, and $c\in\Q$
\If{$f_1\neq c'y$, $c'\in\Q$}
\State apply $f_1\mapsto x$, $y\mapsto y$ to $f$
\Else
\State apply $f_1\mapsto x$, $x\mapsto y$ to $f$
\EndIf
\State write $f=a_0x^4+a_1x^3y+a_2x^2y^2+R$, $a_0,a_1\in\Q$, $a_2\in\Q^{*}$, and $R\in E_5$
\State apply $y\mapsto y-\frac{a_1}{2a_2}x$, $x\mapsto x$ to $f$
\Return $f$
\end{algorithmic}
\end{algorithm}

In the cases $X_9$ and $Y_{r,s}$, in general, a real field extension is required for this step. This leads to an implementational problem for the subsequent steps of the algorithm, since if we represent an algebraic number field as $\Q[z]/\langle m\rangle$, with $m\in \Q[z]$ irreducible, we have to determine which root of $m$ the generator $\overline{z}$ corresponds to.
We discuss how to handle this problem in the cases $X_9$ and $Y_{r,s}$ in Sections \ref{sec:Parab} and \ref{sec:Hyper}, respectively. 

Next we eliminate the remaining monomials in $\supp(f)\setminus\supp(T)$ underneath or on $\Gamma(T)$.
We consider the cases where $\Gamma(T)$ has exactly one face and where $\Gamma(T)$ has two faces separately. 

\subsubsection*{The Newton polygon has only one face} Note that in the case $X_9$, after the linear transformation, there are no monomials in $\supp(f)\setminus\supp(T)$ underneath or on $\Gamma(T)$. In the remaining cases, we have $\jet(T,d)=m_0$, where $m_0$ is a monomial, and $d$ is the maximum filtration of $f$ with respect to the standard degree as before. Then $\jet(f,d)=cm_0$, with $0\neq c\in\Q$. By considering $\frac{1}{c}f$, we may assume that $c=1$. Indeed, the transformations that remove the required terms of $\frac{1}{c}f$ also remove the same terms of $f$.

We first eliminate the monomials in $f$ strictly underneath $\Gamma(T)$ and then  the monomials on $\Gamma(T)$, which are not in $\supp(T)$. Note that, to achieve this goal, also monomials above the Newton polygon will have to be eliminated.

We inductively consider jets of $f$ of increasing standard degree, starting at degree $d+1$, until all monomials strictly underneath $\Gamma(T)$ have been removed. For each degree $j$, the jet is transformed such that $\supp(\jet(f,j))=\supp(T,j)$.
While $\jet(T,j)=m_0$, all terms of degree $j$ in $f$ have been created through a term in homogeneous non-linear polynomials $v_1$ or $v_2$ in 
\begin{eqnarray*}\phi(x)&=& \text{linear terms}+v_1+\text{terms of higher degree than $v_1$}\\
\phi(y)&=& \text{linear terms}+v_2+\text{terms of higher degree than $v_2$}.
\end{eqnarray*} 
Considering Equations (\ref{trans:xn}) and (\ref{trans:yn}), and taking into account that by $c=1$, we may assume $c_x=1$ and $c_x=c_y=1$, respectively, such a transformation maps $m_0=x^\alpha y^\beta$ in $\NF(T)(b)$ to \[x^\alpha y^\beta+\alpha x^{\alpha -1}y^\beta v_1+\beta x^\alpha y^{\beta -1}v_2+\text{terms of higher degree}.\] Since $m_0$ has degree less than $j$ and is the only term in $\NF(T)(b)$ with degree less than or equal to $j$, the terms of degree $j$ in $f$ can be written as 
\[v_1\frac{\partial m_0}{\partial x}+v_2\frac{\partial m_0}{\partial y},\]
where $\deg(v_1)>1$ and $\deg(v_2)>1$. We can eliminate these terms by the transformation $x\mapsto x-v_1$, $y\mapsto y-v_2$. This transformation possibly creates terms of degree greater than $j$, but does not change $\jet(f,j-1)$.

As soon as the $j$-jet of $\NF(T)(b)$ contains more than one term, the situation becomes more complicated. For example, if $j$ is the degree of a monomial in $\supp(T)$, then some of the monomials of degree $j$ in $\supp(f) \setminus\supp(T) $ may result from the linear terms of $\phi(x)$ and $\phi(y)$. However, all monomials in $f$ underneath $\Gamma(T)$ have already  been eliminated before we reach such a case. Indeed, taking Equations (\ref{trans:xn}) and (\ref{trans:yn}) into account, linear terms in $\phi(x)$ will not create any additional monomials and $\phi(y)$ will only create additional monomials above $\Gamma(T)$. Moreover, a case by case analysis shows that terms in $f$ underneath $\Gamma(T)$ also cannot result from higher order terms in $\phi(x)$ and $\phi(y)$:
\begin{itemize}[leftmargin=7mm]

\item If $m_0=x^d$, then the real main type of $f$ is one of the following: $W_{12}$, $W_{13}$, $E_{12}$, $E_{13}$, $E_{14}$, $J_{10}$ (see Figures  \ref{fig exceptional W}, \ref{fig exceptional E}, and \ref{fig parabolic}). 

Considering the normal forms of $W_{12}$ and $W_{13}$, we have $\jet(T,4)=x^4$ and $\jet(T,5)\neq x^4$. Suppose $\jet(f,4)=x^4$. If $f$ is of type $W_{12}$, then $f$ cannot have any monomials underneath $\Gamma(T)$. Let $f$ be of type $W_{13}$ and suppose $y^5\in\supp(f)$. Then it follows from Theorem \ref{thm:quasiNF} that $f$ is of type $W_{12}$ which leads to a contradiction. Hence $y^5\not\in\supp(f)$. 

If $f$ is of type $E_{12}$ or $E_{13}$, then $\jet(T,5)=x^3$ and $\jet(T,6)\neq x^3$. 
After applying the algorithm described above, we may assume that $\jet(f,5)=x^3$. If $y^6\in\supp(f)$, then $f$ is of type $J_{10}$. If $y^7\in\supp(f)$, then $f$ is not of type $E_{13}$, but rather of type $E_{12}$. Hence if $f$ is of type $E_{12}$ or $E_{13}$, then $f$ has no monomials underneath $\Gamma(T)$. Similarly, it can be shown that if $f$ is of type $E_{14}$, then $f$ has no monomials underneath $\Gamma(T)$ after $f$ has been transformed such that $\jet(f,6)=x^3$. 

If $f$ has real main type $J_{10}$ and $\jet(f,3)=x^3$, the only monomials that may occur underneath $\Gamma(T)$ are $xy^3$, $y^4$ and $y^5$. Suppose that $xy^3\in\supp(f)$. Since, by Equation (\ref{trans:xn}), this term cannot result from linear terms of $\phi(x)$ or $\phi(y)$, and since $\jet(f,3)=x^3$, it follows that $xy^3\in\supp(v_1\frac{\partial x^3}{\partial x})=\supp(3v_1x^2)$. Since $x^2\nmid xy^3$, this leads to a contradiction. If $y^4\in\supp(f)$, then $f$ is of real main type $E_6$ by Theorem \ref{thm:quasiNF}, which again gives a contradiction. Similarly, if $y^5\in\supp(f)$, then $f$ is of real main type $E_8$.

\item If $m_0\neq x^d$, then the real main type of $f$ is one of the following: $Z_{11}$, $Z_{12}$, $Z_{13}$ (see Figure~\ref{fig exceptional Z}).

 We show that, in these cases, there are no monomials underneath $\Gamma(T)$ after the linear transformation that transforms $f$ such that $\jet(f,d)=m_0$. Suppose $f$ is of one of the main types under consideration and that $\jet(f,d)=m_0$. If $f$ is of type $Z_{11}$, there are no  monomials of degree greater than $d$ underneath $\Gamma(T)$. Suppose $f$ is of real main type $Z_{12}$ or $Z_{13}$. If $y^5$ or $xy^4$ are elements of $\supp(f)$, then, by Theorem \ref{thm:quasiNF}, $f$ is of real main type $W_{12}$ or $Z_{12}$, respectively, which leads to a contradiction.

\end{itemize}

Therefore, if $m_0\neq x^d$, then after the linear transformation given in Algorithm \ref{alg:1jetEx}, $f$ has no monomials underneath $\Gamma(T)$.  If $m_0=x^d$ and $j$ is such that $\jet(T,j)=m_0$, then one can write
\[\jet(f,j)=m_0+v_1\frac{\partial m_0}{\partial x}+0\frac{\partial m_0}{\partial y}.\]
Hence, iterative application of Algorithm~\ref{alg:Transformation} with $t$ a term of $f$ of standard degree $j$ removes all monomials underneath $\Gamma(T)$. Note that in the algorithm we have $m_x=m_0$, so we make only use of the case in lines \ref{line:OneFace} - \ref{line:OneFaceEnd}. The other case will be used later to eliminate monomials above the diagonal in a similar manner. Moreover, note that Algorithm \ref{alg:Transformation}, as formulated, also works for the polynomial $t=\jet(f,j)-\jet(f,j-1)$.

After eliminating all the monomials in $\supp(f)$ strictly underneath $\Gamma(T)$, we may assume that $\phi_{-1}^w(x)=0$ and $\phi_{-1}^w(y)=0$. 

Next we eliminate the remaining monomials in $\supp(f)\setminus\supp(T)$ on $\Gamma(T)$, see line \ref{line:weighteJet} in Algorithm \ref{alg:ClassPara}. The only case where such monomials can occur is the $J_{10}$ case. In this case, these monomials can theoretically be eliminated by a weighted homogeneous transformation. However, similar to the linear transformation discussed above, this transformation creates implementational difficulties, since it may require real field extensions. We will discuss how to handle this problem in Section \ref{sec:Hyper}.
 
\subsubsection*{The Newton polygon has two faces} In this case, all the normal forms we have to consider are of the form $x^n+x^2y^2+ay^m$ with $m,n\in\N$, see Figure \ref{fig hyperbolic}.

 After a linear transformation (see Algorithm \ref{alg:1jetEx} for the case $J_{10+k}$, Algorithm \ref{alg:1jetX9+k} for the case $X_{9+k}$, and Section \ref{sec:Hyper} for the case $Y_{r,s}$), we may assume that $\supp(\jet(f,d))=\supp(T,d)$, where $d$ is the maximum filtration of $f$ with respect to the standard grading. We first remove all monomials $m$ with $$\displaystyle\max\left\{w_i\dash\deg(m)\mid i\right\}\le d_w$$ which are not in $\supp(T)$, see line \ref{line:TwoFacesEnd} in Algorithm \ref{alg:ClassPara}. The only cases in which such monomials may occur are the $J_{10+k}$ cases. The monomials $y^4$, $y^5$ and $xy^3$ do not occur in $f$ if it is of type $J_{10+k}$, $k>0$, since presence of either of them leads to a different type. Using Algorithm \ref{alg J10k}, we transform $f$ by a weighted linear transformation such that $\supp(w\dash \jet(f,6))=\supp(w\dash\jet(T,6))$ where $w=(2,1)$, which means that $f$ is in the desired form.

   \begin{algorithm}[ht]
\caption{}%
\label{alg J10k}
\begin{algorithmic}[1]
\Require{$f$ of type $J_{10+k}$ with $\supp((2,1)\dash \jet(f,5))=\emptyset$}
\Ensure{$h\in\Q[x,y]$ with $f\sim h$  and with $\supp((2,1)\dash \jet(h,6))=\supp((2,1)\dash\jet(T,6))$}
 \State write $f$ as $f = x^3+cx^2y^2+dxy^4+ey^6+R$ with $c,d,e\in\Q$ and $R\in E_7^{(2,1)}$
\State let $q$ be the double root of $x^3+cx^2+dx+e$
\State apply $x\mapsto x+qy^2$, $y\mapsto y$ to $f$
\Return $f$
\end{algorithmic}
\end{algorithm}

Similar to the case where $\Gamma(T)$ has one face, we may assume that $\coeff(f,x^2y^2)=1$ by considering $\frac{1}{c}f$, where $c=\coeff(f,x^2y^2)$. We now, again, consider jets of $f$ of increasing standard degree, starting at degree $d+1$. However, in this case, we only eliminate those monomials not lying in $\supp(T)$ that are strictly underneath or on $\Gamma(T)$. Suppose we are considering the standard weighted $j$-jet. We first eliminate terms of the form $c_1xy^{j-1}$ and $c_2yx^{j-1}$ which lie underneath or on $\Gamma(T)$, using permissible $x$- and $y$-segments, that is, by applying the transformation \[x\mapsto x+ \frac{c_1}{2} y^{j-3},\quad y\mapsto y\]
 or
 \[x\mapsto x,\quad y\mapsto y+ \frac{c_2}{2}x^{j-3},\]
 respectively. This process is described in Algorithm \ref{alg:Transformation} with $t=c_1xy^{j-1}$ or $t=c_2yx^{j-1}$, respectively.
   Now, if $n,m\neq j$, the monomials $x^j$ and $y^j$ do not occur in $\supp(f)$: Otherwise, by Lemma \ref{principalpart} and Theorem \ref{thm:quasiNF}, $f$ would be of a different type than assumed. 
   No other monomials can occur in $\supp(f)\setminus\supp(T)$ underneath or on $\Gamma(T)$.

\begin{algorithm}[ht]
\caption{}%
\label{alg:Transformation}
\begin{algorithmic}[1]
\Require{$f\in\K [x,y]$ of corank $2$, modality $1$ and type $T$, and a term $t\notin \supp(T)$ of $f$ such that either\smallskip
\begin{enumerate}[leftmargin=7mm]
\item[(a)]$t\in \Jac(f)$, $|\supp(T,j)|=1$, $\supp(\jet(f,j-1))=\supp(T,j-1)$ where $j=\deg(t)$,
\item[(b)]$t$ is in a permissible chain, and $\supp(\jet(f,d))=\supp(T,d)$ where $d$ is the maximum filtration of $f$, or
\item[(c)]
$t\in \Jac(f)$, $j:=w(T)\dash\deg(t)>d(T)$, $\supp(w(T)\dash\jet(f,j-1))=\supp(w(T)\dash\jet(T,j-1))$, and $j$ is smaller than the $w(T)$-degree of the unique element in Arnold's system.
\end{enumerate}

}
\Ensure{$g\in\K [x,y]$ such that $f\sim g$ and $\supp(g)\cap\supp(t)=\emptyset$ 
}
\State let $f_0$ be piecewise quasihomogeneous part of $f$ of degree $d(T)$ w.r.t.~$w(T)$
\If{$t\in\Jac(f)$}
\State $(u_1,u_2):=((1,1),(1,1))$
\Else
\State $(u_1,u_2):=w(T)$
\EndIf
\State let $m_x$ be the term of $\frac{\partial f_0}{\partial x}$ of lowest $u_2$-degree 
\State let $m_y$ be the term of $\frac{\partial f_0}{\partial y}$ of lowest $u_1$-degree
\If{$m_x|t$}\label{line:OneFace}
\State \vspace{-4mm}\label{line:OneFaceEnd}\begin{eqnarray*}
\alpha:\K[x,y]&\rightarrow&\K[x,y]\\
       x&\mapsto& x-t/m_x\\
       y&\mapsto& y
\end{eqnarray*}
\Else 
\State \vspace{-4mm}\begin{eqnarray*}
\alpha:\K[x,y]&\rightarrow&\K[x,y]\\
       x&\mapsto& x\\
       y&\mapsto&y-t/m_y
 \end{eqnarray*}
\EndIf
\State $g:=\alpha(f)$
\Return $g$
\end{algorithmic}
\end{algorithm}

\subsection*{Step (III): Elimination of monomials above the Newton polygon} In the case where $\Gamma(T)$ has two faces, Theorem \ref{principalpart} allows us to delete all monomials above $\Gamma(T)$.

 We now consider the case where $\Gamma(T)$ has one face. Here we use the algorithmic method that was use in the proof of Theorem \ref{thm:quasiNF} to eliminate the monomials above $\Gamma(T)$ in $\Jac(f)$. 
Let $\{m_1\}$ be the Arnold system for the singularity under consideration. Let $d'_w$ be the degree of $m_1$. Again we iteratively consider the $w$-degree $j$ part of $f$ in $\Jac(f)$, for increasing $j$  with $d_w < j\leq d'_w$. This polynomial can be written as
\[\frac{\partial f_0}{\partial x}v_1+\frac{\partial f_0}{\partial y}v_2,\]
where $f_0$ is the quasihomogeneous part of $f$ and $v_1,v_2\in\R[x,y]$. After applying the transformation $x\mapsto x-v_1$, $y\mapsto y-v_2$ to $f$, the $w$-degree $j$ part has no terms of degree $j$ in $\Jac(f)$. Note that for $d_w< j<d'_w$, all terms of $w$-degree $j$ are in $\Jac(f)$.  In all the cases we consider, each weighted diagonal above $\Gamma(T)$ of $w$-degree less than or equal to $ d'_w$ will only contain one monomial in $\Jac(f)$. Hence, either $v_1$ or $v_2$ are zero. We therefore can apply Algorithm \ref{alg:Transformation} to remove this monomial.

The remaining $w$-degree $d'_w$ part of $f$ can be written as
\[\frac{\partial f_0}{\partial x}v_1+\frac{\partial f_0}{\partial y}v_2+cm_1,\]
where $v_1,v_2\in\R[x,y]$ and $c\in\R$. Applying the transformation $x\mapsto x-v_1$, $y\mapsto y-v_2$ results in transforming the part of $f$ of weighted degree $d'_w$ to $cm_1$. A case by case analysis shows that we can always choose $v_2=0$. See line \ref{line:write in Arnold system} in Algorithm \ref{alg:ClassPara}.

Since all terms of $w$-degree greater than $d'_w$ are above the weighted determinacy, they can be deleted.

\begin{remark}
The question may be asked why we change the strategy, in the sense that we iteratively consider the terms of increasing standard degree in the process of removing all terms underneath $\Gamma(T)$ and the terms of increasing $w$-degree above $\Gamma(T)$. The reason is that canceling terms underneath $\Gamma(T)$ can only be done using transformations of negative filtration with respect to $w$, since the terms underneath $\Gamma(T)$ have smaller $w$-degree than the terms on $\Gamma(T)$. This means that in the process of deleting monomials of a given $w$-degree, monomials of lower and higher $w$-degree may be created. Since the filtration of a transformation with respect to the standard degree cannot be negative, this is not the case when canceling terms by increasing standard degree via the above methods. However, terms above $\Gamma(T)$ are canceled using transformations with non-negative filtration with respect to $w$ via the above method and therefore we consider terms by increasing $w$-degree in this case, which leads to a much simpler process.
\end{remark}

\begin{algorithm}[p]
\caption{Classification and Determination of Parameter: General Structure}%
\label{alg:ClassPara}
\begin{algorithmic}[1]

\Require{$f \in \m^3\subset\Q[x,y]$, a germ of modality $1$ and corank $2$}

\Ensure{the real singularity types and normal forms of $f$, as well as the corresponding normal form equations in the right equivalence class of $f$ with the respective parameter given as the unique root of its minimal polynomial over $\Q$ in a specified interval}
\vspace{0.2cm}

\noindent\emph{I. Determine  the complex type:} 
\State compute the complex type $T$ of $f$ using {\tt classify()} from \citep{classify2}
\vspace{0.2cm}

\noindent\emph{II. Eliminate the monomials in $\supp(f)$ underneath or on $\Gamma(T)$:}
\State let $d$ be the maximum filtration of $f$ w.r.t.~the standard grading

\State\label{line:linearCoordinateChange}apply a linear coordinate change to $f$ such that $\supp(\jet(f,d))=\supp(T,d)$ 
\State $w:=(w_1,\ldots,w_n):=w(T)$ 
\State $d_w:= d(T)$
\If{$n=1$}
	\For{$j=d+1,d+2,\ldots$ \textbf{while} $\supp(f)\cap(\R^2\setminus\Gamma_+(T))\neq\emptyset$}\label{line:belowDiagonalStart}
			\While{$f$ has a term $t$ of standard degree $j$}
					\State replace $f$ by the output of Algorithm \ref{alg:Transformation} applied to $f$ and $t$ 			\EndWhile
	\EndFor\label{line:belowDiagonalEnd}			

	\State \label{line:weighteJet}apply a weighted homogeneous transformation to $f$ such that\vspace{-1mm} \[\supp(w\dash\jet(f,d_w))=\supp(w\dash\jet(T,d_w))\]

\Else
      \State \label{line:TwoFacesEnd}apply a $w_2$-weighted homogeneous transformation to $f$ such that\vspace{-1mm} \[\supp(w_2\dash\jet(f,d_w))=\supp(w_2\dash\jet(T,d_w))\]
	\For{$j= d+1, d+2,\ldots$ \textbf{while} $\supp(f)\cap(\overline{\R^2\setminus\Gamma_+(T)})\neq\emptyset$}\label{line:n2forloop}
	
				\While{$f$ has a term $t$ of the form $bxy^{j-1}$ or $byx^{j-1}$ underneath or on $\Gamma(T)$}\label{line:TwoFacesStart}
							\State \label{line:TwoFacesEnd0}replace $f$ by the output of Algorithm~\ref{alg:Transformation} applied to $f$ and $t$
			\EndWhile

	\EndFor
\EndIf
\vspace{0.2cm}

\noindent\emph{III. Eliminate the monomials above $\Gamma(T)$ which are not in $\supp(T)$:}
\If{$n=1$}
    \If{$w\dash\jet(\Gamma(T),d_w)$ is non-degenerate}
	\State let $d'_w$ be the $w$-degree of the unique element in Arnold's system 
	\For{$j=d_w+1,\ldots, d'_w$}
	\While{$f$ has a term $t\in\Jac(g)$ of $w$-degree $j$}
		\State replace $f$ by the output of Algorithm~\ref{alg:Transformation} applied to $f$ and $t$
	\EndWhile 
	\EndFor
	\State $f:=w\dash\jet(f,d'_w)$
	\State modulo $\Jac(f)$, write the sum of the terms of $f$ above $\Gamma(T)$ in Arnold's system\label{line:write in Arnold system}
    \Else
      \State replace $f$ by the output of Algorithm~
              \ref{alg:Yrabove} applied to $f$
    \EndIf
\Else 
	\State $f:=w\dash\jet(f,d_w)$
\EndIf
\vspace{0.2cm}

\noindent\emph{IV. Scale and read off the desired information:}
\State apply a transformation of the form $x\mapsto\pm x$, $y\mapsto\pm y$ to $f$ such that the signs of the coefficients of the non-moduli monomials match
   the signs of one possible real normal form $F$ in Table~\ref{tab:normal_forms}
\State read off the corresponding type of $F$
\State\label{line:scaling}apply a transformation of the form $x\mapsto \lambda_1 x$, $y\mapsto \lambda_2 y$ with $\lambda_1,\lambda_2> 0$ in a real algebraic extension of $\Q$ to transform the non-moduli terms to the terms of $F$
\State determine the minimal polynomial over $\mathbb{Q}$ for the parameter in $F$
\State using \cite{realclassify2}, determine all possible types and corresponding parameters
\Return all types, normal forms, and parameters
\end{algorithmic}
\end{algorithm}

\subsection*{Step (IV): Scaling and reading off the desired information.} 
Lastly we scale $f$ such that the coefficients of the non-moduli terms coincide with those in one of the real subtypes of the real main type of $f$, read off the needed information and determine all normal form equations in the equivalence class of $f$, using \cite{realclassify2}.

\section{Implementing the Exceptional Cases}\label{sec:Excep}

For the exceptional singularities (see Figures \ref {fig exceptional W}, \ref{fig exceptional E}, and \ref{fig exceptional Z}), the general algorithm can be implemented in a straightforward manner. As an example, we consider singularities of real main type $E_{14}$, and give all details in Algorithm \ref{alg:E_14}. In this case, we have $w=(8,3)$, $d=3$, $d_w=24$, $d'_w=26$, $m_0=x^3$, $m_1=xy^6$ and $f_0=x^3+y^8$ in the general algorithm. Algorithm \ref{alg:E_14} corresponds directly to our implementation of the respective case in \textsc{Singular}. Note that, although the scaling step in line \ref{line:scaling} of Algorithm \ref{alg:ClassPara} has not been implemented explicitly, it has been taken into account in Algorithm \ref{alg:E_14} when determining the minimal polynomial. We use the following notation.

\begin{notation}
Suppose $p\in\Q[z]$ is a univariate polynomial over the rational numbers, and $I\subset\R$ is an interval such that there is exactly one $a \in I$ with $p(a)=0$. We then denote the monic irreducible divisor of $p$ with root $a$ by $m_I(p)$.  
\end{notation}
When denoting real subcases, we identify $+$ with $+1$ and $-$ with $-1$. For example, we use the notation $E_{14}^{\sign(b)}$ for $E_{14}^{+}$ if $b>0$, and $E_{14}^{-}$ if $b<0$.\bigskip

\begin{figure}[ht]
\begin{center}
\setlength{\tabcolsep}{3mm}
\begin{tabular}
[c]{ccc}%
{\includegraphics[
height=1.86in,
]%
{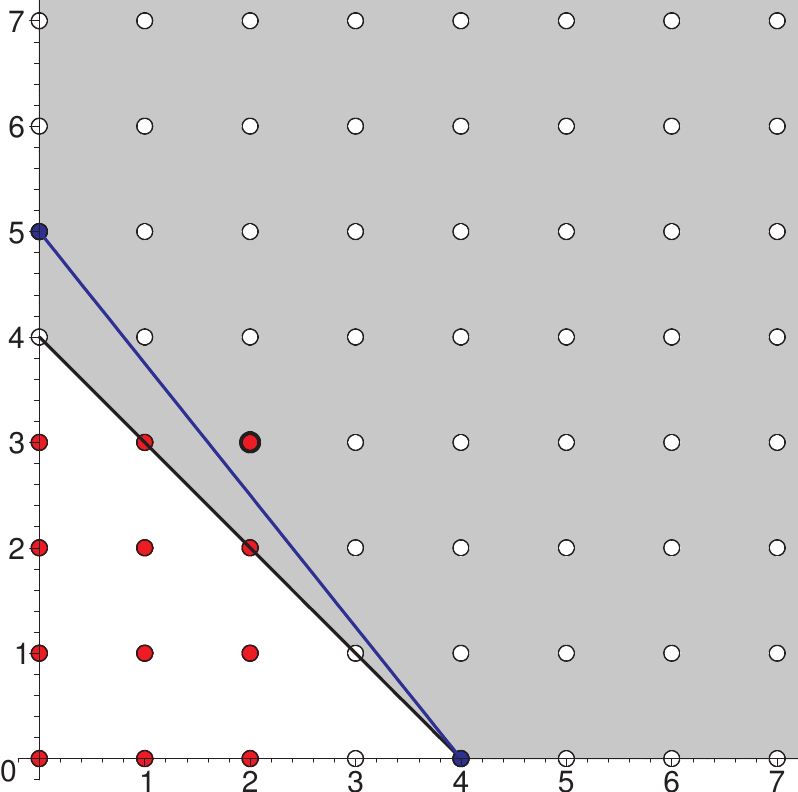}%
}%
& {\includegraphics[
height=1.86in,
]%
{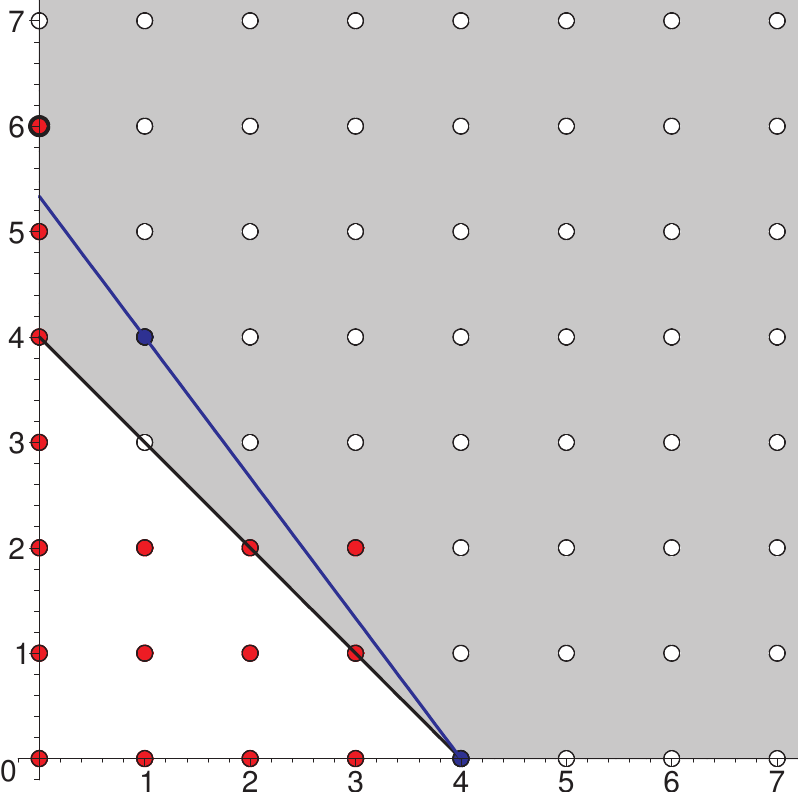}%
}\\
$W_{12}$ & $W_{13}$ &
\end{tabular}
\caption{Exceptional Singularities of type W}%
\label{fig exceptional W}%
\end{center}
\end{figure}

\begin{figure}[ht]
\begin{center}
\setlength{\tabcolsep}{3mm}
\begin{tabular}
[c]{ccc}%
{\includegraphics[
height=2.1in,
]%
{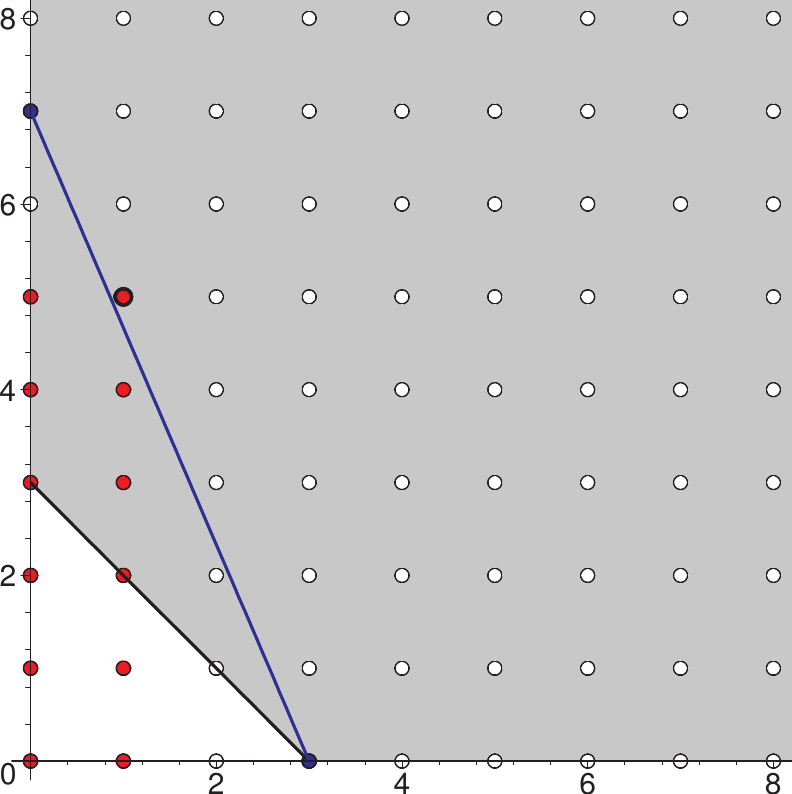}%
}%
&
{\includegraphics[
height=2.1in,
]%
{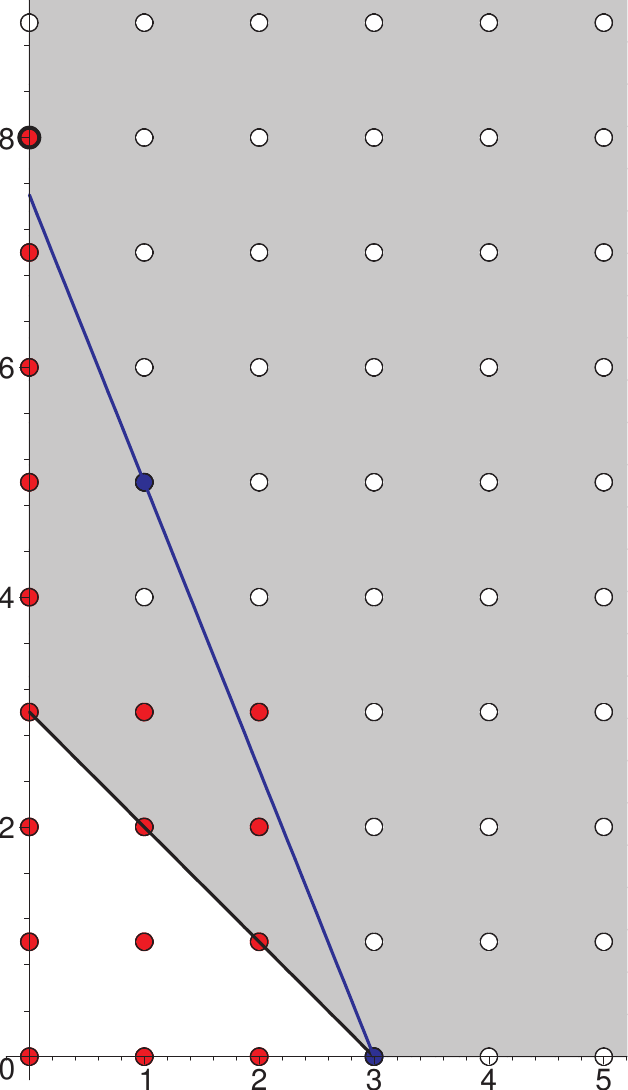}%
}%
&
{\includegraphics[
height=2.1in,
]%
{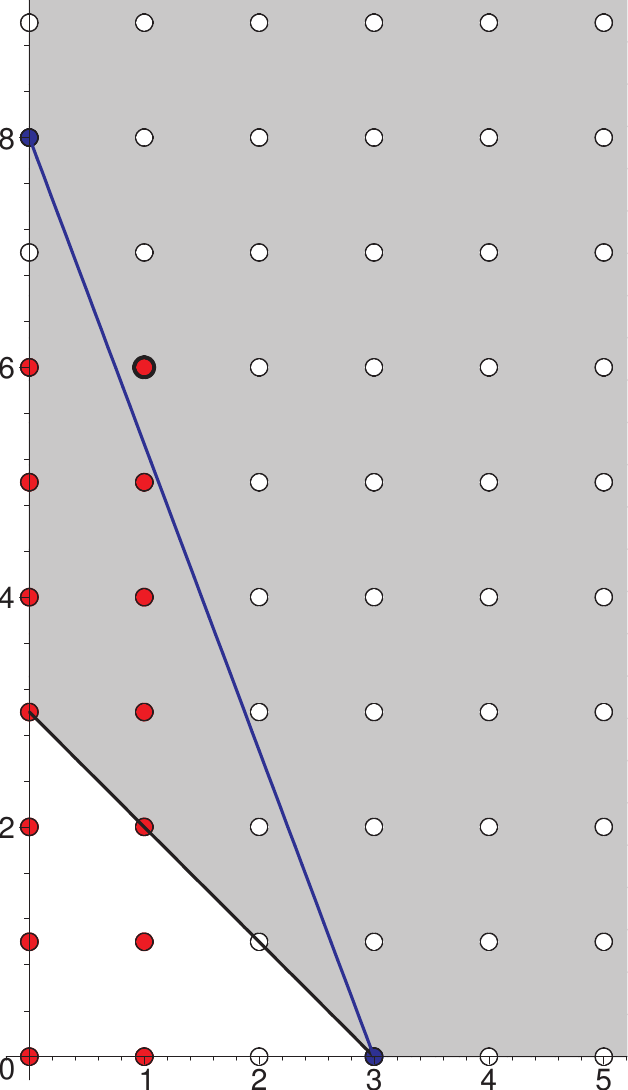}%
}%
\\
$E_{12}$ & $E_{13}$ & $E_{14}$
\end{tabular}
\caption{Exceptional Singularities of type E}%
\label{fig exceptional E}%
\end{center}
\end{figure}

\begin{algorithm}[ht]
\caption{Algorithm for the case $E_{14}$}%
\label{alg:E_14}
\begin{algorithmic}[1]

\Require{$f \in \m^3\subset\Q[x,y]$ of complex singularity type $T=E_{14}$}

\Ensure{
all normal form equations in the right equivalence class of $f$, each specified as a tuple of the real singularity type, the normal form, the minimal polynomial of the parameter, and an interval such that the parameter is the unique root of the minimal polynomial in this interval}
\vspace{1mm}

\noindent\textit{II. Eliminate the monomials in $\supp(f)$ underneath or on $\Gamma(T)$:}
\vspace{0.2cm}

\noindent \textit{Apply Algorithm \ref{alg:1jetEx}:}
\vspace{0.2cm}

\State $h := \jet(f,3)$
\State factorize $h$ as $ch_1^3$, where $h_1$ is homogeneous of degree one and $c\in\Q^*$
\If{$h\neq c'y, c'\in\Q$}
\State apply $h_1\mapsto x$, $y\mapsto y$ to $f$
\Else
\State apply $h_1\mapsto x$, $x\mapsto y$ to $f$
\EndIf
\State $g:=\frac{1}{c}f$

\vspace{0.2cm}
\noindent\textit{Consider the terms of $g$ of standard degree $j=4$ and apply Algorithm \ref{alg:Transformation}:}
\vspace{0.2cm}

\State $h_1:= \frac{\jet(g,4)-x^3}{3x^2}$
\State apply $x\mapsto x-h_1$, $y\mapsto y$ to $g$

\vspace{0.2cm}
\noindent\textit{Consider the terms of $g$ of standard degree $j=5$ and apply Algorithm \ref{alg:Transformation}:}
\vspace{0.2cm}

\State $h_2:=\frac{\jet(g,5)-x^3}{3x^2}$
\State apply $x\mapsto x-h_2$, $y\mapsto y$ to $g$

\vspace{0.2cm}
\noindent\textit{III. Eliminate the monomials above $\Gamma(T)$ which are not in $\supp(T)$:}
\vspace{0.2cm}
\State $w:=(8,3)$

\vspace{0.2cm}
\noindent\textit{Consider the terms of $g$ of $w$-degree $j=25$ and apply Algorithm \ref{alg:Transformation}:}
\vspace{0.2cm}

\State $h'_1:=\frac{w\dash\jet(g,25)-w\dash\jet(g,24)}{3x^2}$
\State $x\mapsto x- h'_1$, $y\mapsto y$
\vspace{0.2cm}

\State $f:=cg$\vspace{2mm}

\noindent\textit{IV. Read off the desired information:}
\vspace{0.2cm}

\If{$c<0$}
\State apply $x\mapsto -x$, $y\mapsto y$ to $f$
\EndIf
\State write $f$ as $f = cx^3+by^8+R$\quad with $c\in\Q_{>0}$, $b\in\Q^{*}$,\ and $R\in E_{25}^{(8,3)}$
\State $t:= \operatorname{coeff}(f,xy^6)$
\State $p:= z^{12}-c^{-4}|b|^{-9}t^{12}\in\mathbb{Q}[z]$

\If{$t>0$}
\Return$ (E_{14}^{\operatorname{sign}(b)},\ x^3+\operatorname{sign}(b)\cdot y^8+axy^6,\  \operatorname{m}_{(0, \infty)}(p),\ (0, \infty) )$
\Else
\Return$ (E_{14}^{\operatorname{sign}(b)},\ x^3+\operatorname{sign}(b)\cdot y^8+axy^6,\  \operatorname{m}_{(-\infty,0]}(p),\ (-\infty,0] )$
\EndIf
\end{algorithmic}
\end{algorithm}

\begin{figure}[t]
\begin{center}
\setlength{\tabcolsep}{3mm}
\begin{tabular}
[c]{ccc}%
{\includegraphics[
height=1.6336in,
]%
{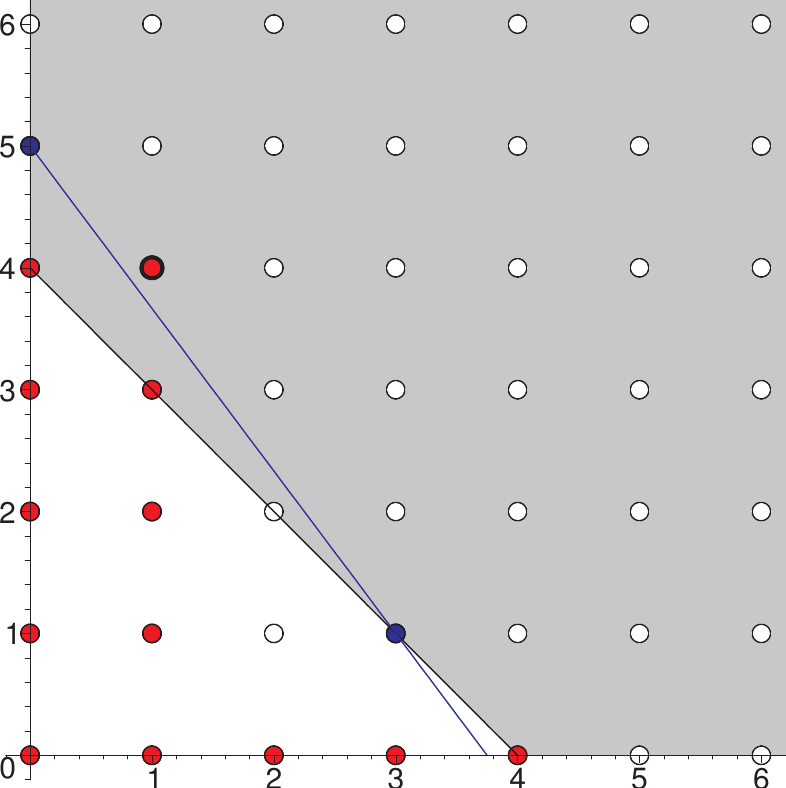}%
}%
&
{\includegraphics[
height=1.6336in,
]%
{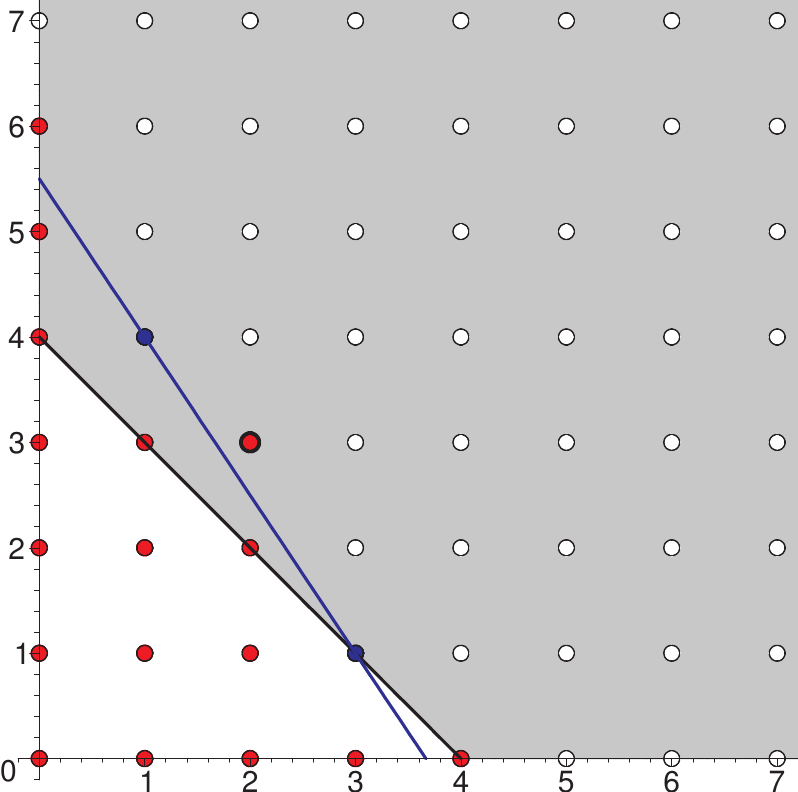}%
}%
&
{\includegraphics[
height=1.6336in,
]%
{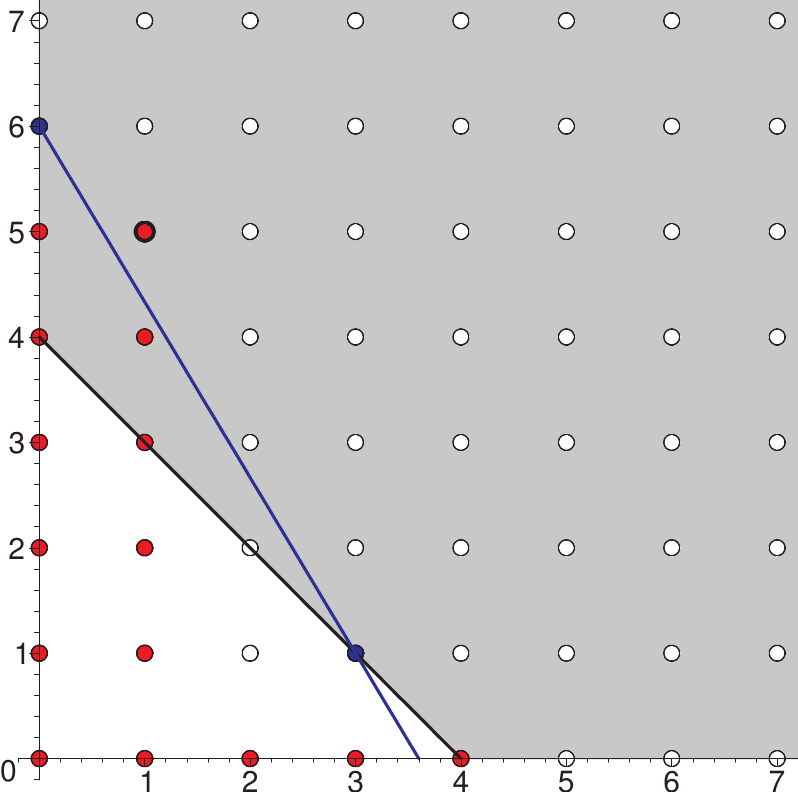}%
}%
\\
$Z_{11}$ & $Z_{12}$ & $Z_{13}$
\end{tabular}
\caption{Exceptional Singularities of type Z}%
\label{fig exceptional Z}%
\end{center}
\end{figure}

\phantom{XXXXXXXXXXXXXXXXXXXXXXXXXXXXXXXXXX}

\phantom{XXXXXXXXXXXXXXXXXXXXXXXXXXXXXXXXXXXXXXXX}

\phantom{XXXXXXXXXXXXXXXXXXXXXXXXXXXXXXXXXXXXXX}

\phantom{XXXXXXXXXXXXXXXXXXXXXXXXXXXXXXXXXXXXX}

\section{Parabolic Singularities}\label{sec:Parab}

In this section, we give details on the application of the general algorithm (Algorithm \ref{alg:ClassPara}) for the parabolic singularities. See Figure \ref{fig parabolic} for these cases.

\begin{figure}[b]
\begin{center}
\setlength{\tabcolsep}{3mm}
\begin{tabular}
[c]{ccc}%
{\includegraphics[
height=2.4in,
]%
{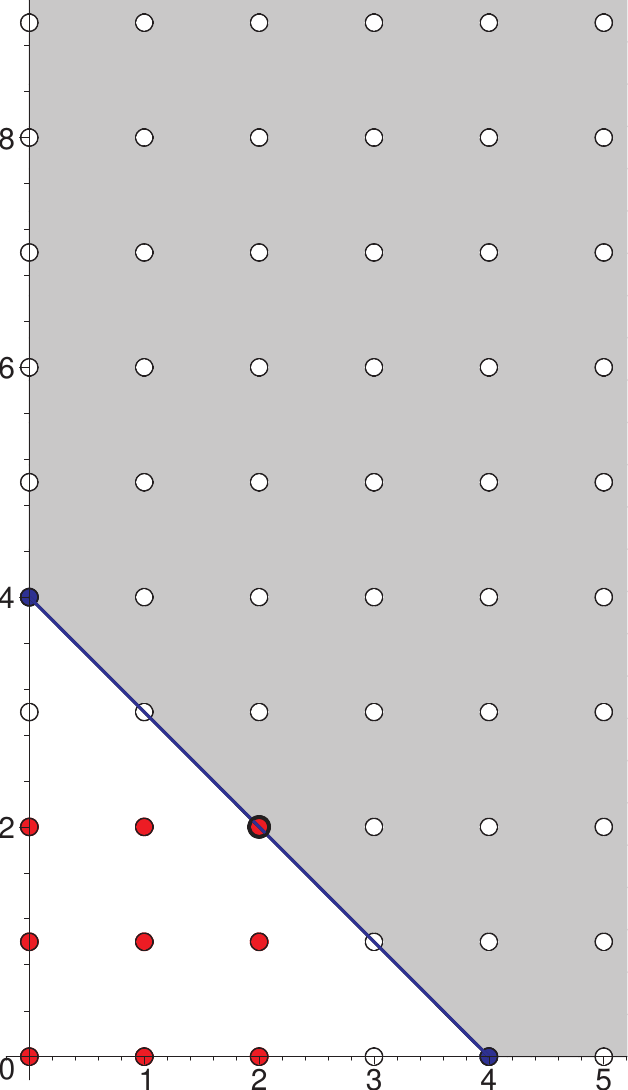}%
}%
&
{\includegraphics[
height=2.4in,
]%
{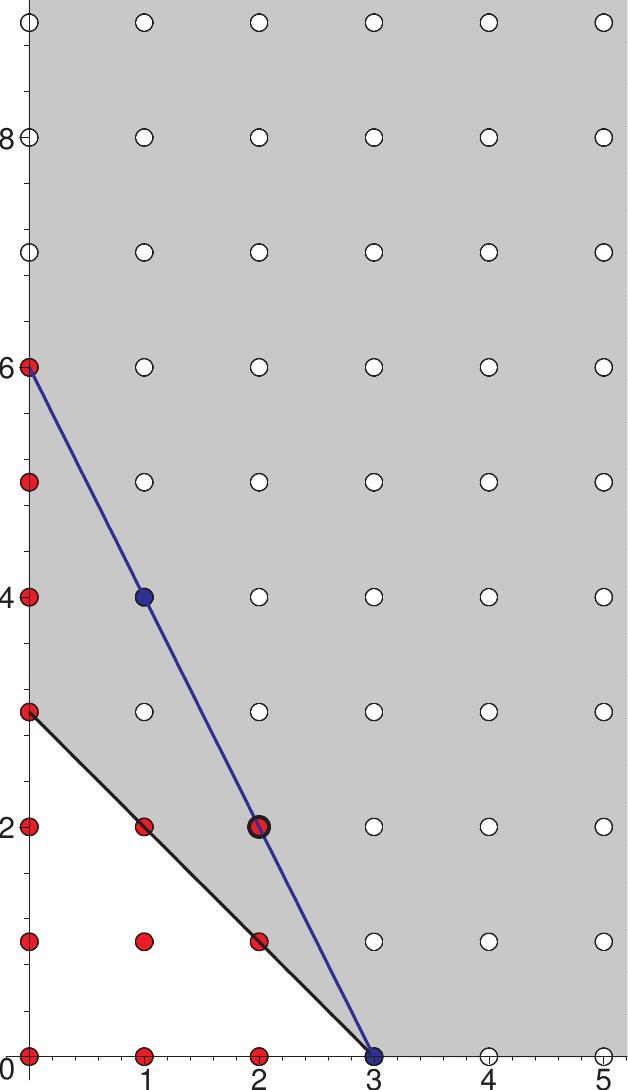}%
}%
\\
$X_{9}$ & $J_{10}$ 
\end{tabular}
\caption{Parabolic Singularities}%
\label{fig parabolic}%
\end{center}
\end{figure}

\subsection{The \texorpdfstring{$\boldsymbol{J_{10}}$}{J10} case}
In this case we have $w=(6,3)$ and $d_w=d'_w=18$. As discussed in Section~\ref{section:generalAlg}, lines \ref{line:belowDiagonalStart} - \ref{line:belowDiagonalEnd} in the general algorithm are redundant. Moreover, by $d_w=d'_w$, Step (III) is redundant. Unfortunately we may need a real field extension in line \ref{line:weighteJet} of the general algorithm. This leads to an implementational problem for the subsequent steps of the algorithm, since if we represent an algebraic number field as $\Q[z]/m$, with $m$ irreducible, we have to determine which root of $m$ the generator $\overline{z}$ corresponds to. In this section, we will work around this problem by presenting a method to read off the required information without explicitly doing the required weighted linear transformation, see Algorithm \ref{alg:J_{10}}.

In \cite{realclassify2} it has been shown that in the case of real main type $J_{10}$, the equivalence class of $f$ either contains exactly one normal form equation of type $J_{10}^{+}$ or it contains exactly three normal form equations, two of which are of type $J_{10}^{+}$ and one is of type $J_{10}^{-}$. This shows, in particular, that the real subtype $J_{10}^{-}$ is redundant. The following algorithmic approach to the classification also reconfirms these findings.

\begin{algorithm}[htp]
\caption{Algorithm for the case $J_{10}$}%
\label{alg:J_{10}}
\begin{algorithmic}[1]

\Require{$f \in \m^3\subset\Q[x,y]$ of complex singularity type $J_{10}$}

\Ensure{
all normal form equations in the right equivalence class of $f$, each specified as a tuple of the real singularity type, the normal form, the minimal polynomial of the parameter, and an interval such that the parameter is the unique root of the minimal polynomial in this interval
}
 
\State $f:=\jet(f,6)$
\If{$\operatorname{coeff}(f,x^3)=0$}
\State apply $x\mapsto y$, $y\mapsto x$ to $f$
\EndIf
\State $h:=\jet(f,3)$
\State factorize $h$ as $bg_1^3$\ with $b\in\Q$, $b>0$, and $g_1$ homogeneous of degree $1$
\State apply $x\mapsto g_1$, $y\mapsto y$ to $f$
\State $f:= \frac{1}{\operatorname{coeff}(f,x^3)}f$
\State write $f$ as $f = x^3+ax^2y^2+dxy^4+ey^6$ with $a,d,e\in\Q$
\State apply $x\mapsto x-\frac{a}{3}y^2$, $y\mapsto y$ to $f$
\State write $f$ as $f=x^3+dxy^4+ey^6$ with  $d,e\in\Q$
\State $p^{(\sigma)} := \sigma (4d^3+27e^2) z^6 + (-36d^3-243e^2) z^4
  + \sigma (81d^3+729e^2)z^2 - 729e^2 \in \Q[z]$
\State $k:= s^3+ds+e\in\mathbb{Q}[s]$
\If{$k$ has exactly one real root}
\If{$e<0$}
\Return $(J_{10}^+,\ x^3+ax^2y^2+xy^4,\ \operatorname{m}_{(0, \infty)}(p^+),\ (0, \infty) )$
\EndIf

\If{$e>0$}
\Return $(J_{10}^+,\ x^3+ax^2y^2+xy^4,\ \operatorname{m}_{(-\infty,0)}(p^+),\ (-\infty,0) )$
\EndIf

\If{$e=0$}
\Return $(J_{10}^+,\ x^3+ax^2y^2+xy^4,\ z,\ [0,0] )$

\EndIf
\Else
\If{$k(0)=0$}
\State $C_1 := (J_{10}^+,\  x^3+ax^2y^2+xy^4,\ z^2-\frac{9}{2},\  (-\infty,0))$
\State $C_2 := (J_{10}^-,\  x^3+ax^2y^2-xy^4,\ z,\  [0,0])$\vspace{1mm}
\State $C_3 := (J_{10}^+,\  x^3+ax^2y^2+xy^4,\ z^2-\frac{9}{2},\  (0,\infty))$

\Else
\State $\varepsilon:=1$
\Do
\State replace $\varepsilon$ by a rational number in the interval $(0,\frac{\varepsilon}{2})$
\State \parbox[t]%
  {\linewidth-\algorithmicindent-\algorithmicindent-\algorithmicindent}{
    let $(z_1, \ldots, z_4) \in \Q^4$ be approximations of the four distinct
    real roots of $p^+$ in increasing order with error smaller than
    $\varepsilon$
  }
\State $(I_1,\ldots,I_4):=((z_1-\varepsilon,z_1+\varepsilon),\ldots,(z_4-\varepsilon,z_4+\varepsilon))$
\doWhile{$p^+$ has more than one real root in any of the intervals $I_1,\ldots,I_4$}
\If{$e<0$}
\State $C_1 := (J_{10}^-,\ x^3+ax^2y^2-xy^4,\ \operatorname{m}_{(-\infty,0)}(p^-),\ (-\infty, 0))$
\State $C_2:=(J_{10}^+,\ x^3+ax^2y^2+xy^4,\ \operatorname{m}_{I_1}(p^+),\ I_1)$
\State $C_3:=(J_{10}^+,\ x^3+ax^2y^2+xy^4,\ \operatorname{m}_{I_3}(p^+),\ I_3)$
\Else
\State $C_1 := (J_{10}^-,\ x^3+ax^2y^2-xy^4,\ \operatorname{m}_{(0,\infty)}(p^-),\ (0,\infty))$
\State $C_2:=(J_{10}^+,\ x^3+ax^2y^2+xy^4,\ \operatorname{m}_{I_2}(p^+),\ I_2)$
\State $C_3:=(J_{10}^+,\ x^3+ax^2y^2+xy^4,\ \operatorname{m}_{I_4}(p^+),\ I_4)$
\EndIf
\EndIf
\Return $C_1,\ C_2,\ C_3$
\EndIf
\end{algorithmic}
\end{algorithm}

Let $f\in\m^3$ be a polynomial of complex type $J_{10}$ such that $f\in E_{18}^{(6,3)}$, that is, it is of the form
\[f=cx^3+bx^2y^2+dxy^4+ey^6+R,\qquad R\in E_{19}^{(6,3)}.\] 
Since $f$ is weighted $18$-determined, $f\sim f-R$. Hence, we may assume that
\[f= cx^3+bx^2y^2+dxy^4+ey^6.\]
If $c<0$, then we apply the transformation
\[x\mapsto -x,\qquad y\mapsto y.\]
Since $f$ is weighted homogeneous, a rescaling of the coordinates achieves that $c=1$.
By applying
\begin{equation}\label{Cancelling2ndTerm}
x\mapsto x-\frac{b}{3}y^2,\quad y\mapsto y,
\end{equation}
$f$ is transformed to a polynomial of the form
\[f=x^3+dxy^4+ey^6.\] Since, on the other hand, for $a',d'\in \R$ we have
\begin{equation}\label{phi'}
x^3+a'x^2y^2\pm |d'|xy^4\sim x^3+ax^2y^2\pm xy^4,\quad
a=\frac{a'}{\sqrt{|d'|}},
\end{equation} via the 
invertible transformation \[\alpha' : x\mapsto x,\, y\mapsto \frac{1}{\sqrt[4]{|d'|}}y\] (as applied in Step (IV) of the general algorithm), 
the problem is reduced to finding a transformation $\alpha$ such that $f=x^3+dxy^4+ey^6$ is mapped  to 
\begin{equation}\label{form}
\alpha(f)=x^3+a'x^2y^2\pm |d'|xy^4,\quad a',d'\in \R.
\end{equation}
The composition of $\alpha$ with the transformation in (\ref{Cancelling2ndTerm}) and the scaling of the coordinates, as described above, is then a suitable transformation to be applied in line \ref{line:weighteJet} of the general algorithm.

In the following, we describe a method to determine all polynomials as in (\ref{form}) in the equivalence class of $f$.
Since $\alpha$ has to be weighted homogeneous, it is of the form
\begin{equation}\label{equ form of alpha}
\alpha: x\mapsto x+cy^2,\qquad y\mapsto ty 
\end{equation}
with appropriate $t,c\in\R,\, t\neq 0$. We now determine all possible values for $t$ and $c$.
Applying a transformation as in (\ref{equ form of alpha}) to $f$, we obtain
\begin{equation}\label{eqc}
\alpha(f)=x^3+3cx^2y^2+(3c^2+t^4d)xy^4+(c^3+t^4dc+et^6)y^6.
\end{equation}
Taking $c'=\frac{c}{t^2}$, we can rewrite (\ref{eqc}) as
\begin{equation}\label{eqc'}
\alpha(f)=x^3+3t^2c'x^2y^2+t^4(3c'^2+d)xy^4+t^6(c'^3+dc'+e)y^6.
\end{equation}
Clearly, for a fixed value $t\neq 0$, $c'$ is any real root of
$k:=s^3+ds+e\in\mathbb{Q}[s]$. Due to the scaling of $y$ by $\alpha'$, we may assume that $t=1$ and $c=c'$.

To determine all possible normal form equations in the equivalence class of $f$, we have to consider the following cases:
\begin{itemize}[leftmargin=10mm]
\item[(i)]$k(s)$ has one real root;
\item[(ii)]$k(s)$ has three real roots.
\end{itemize}

Denote by $c_1$, $c_2$ and $c_3$ the complex roots of $k(s)$. Then
\begin{equation}\label{RelationOfRoots}
c_1+c_2+c_3=0,\quad c_1c_2+c_1c_3+c_2c_3=d, \quad\text{and}\quad -c_1c_2c_3=e.
\end{equation}

(i) Let $c=c_1$ be the real root and let $c_2$ and $c_3$ be the complex conjugate roots of $k(s)$. It follows from (\ref{RelationOfRoots})
that $c_2c_3=d+c_1^2$. Since the product of two non-zero complex conjugates is positive,
$d+c_1^2>0$, which implies $d+3c_1^2>0$. Hence, considering
(\ref{eqc'}), $f$ is of type $J_{10}^+$. Since the real root $c_1$ is uniquely determined, the transformation $\alpha$ as considered above  and the real subtype are also uniquely determined. In particular, $a=\frac{3c_1}{\sqrt{3c_1^2+d}}$ is uniquely determined. Define $a_1:=a$, $a_2:=\frac{3c_2}{\sqrt{3c_2^2+d}}$, and $a_3:=\frac{3c_3}{\sqrt{3c_3^2+d}}$. Note that ${da_j^2}=c_j^2(9-3a_j^2)$. Since $c_1$, $c_2$, and $c_3$ are  the roots of $k(s)$, they are also roots of
\[\tilde k(s):=-k(s)\cdot k(-s)=s^6+2ds^4+d^2s^2-e^2.\]
Multiplying $\tilde k(s)$ with $(9-3a_j^2)^3$, we see that $da_j^2$ is a root of 
\[\tilde h_j(w)=w^3+2dw^2(9-3a_j^2)+d^2w(9-3a_j^2)^2-e^2(9-3a_j^2)^3\in\mathbb{Q}[w],\]
that is, $\pm a_j$, $j=1,2,3$, are the roots of
\[h(z):=(4d^3+27e^2)z^6+(-36d^3-243e^2)z^4+(81d^3+729e^2)z^2-729e^2\in\mathbb{Q}[z].\]

By $e=-c_1c_2c_3=0$ and $c_2,c_3\neq 0$, it follows that $$e=0 \Leftrightarrow c_1=0 \Leftrightarrow  a=0.$$

We show that this is the case if and only if  $a_2\in\R$ or $a_3\in\R$. Assume that, without loss of generality, $a_2$ is real. Then $a_2^2=\frac{9c_2^2}{3c_2^2+d}$ is real, which implies that $3c_2^2+d=\lambda c_2^2$ for some $\lambda\in\R$. Therefore $d=\lambda' c_2^2$ for some $\lambda'\in\R$. Since $d\in\R$, it follows that $c_2^2\in\R$. Because $c_2\in\C\setminus\R$, it follows that $c_2=\gamma i$ for some $\gamma\in\R$, hence $c_3=-\gamma i$. This implies that $a_3$ is real, since $d=c_2c_3=\gamma^2$. By (\ref{RelationOfRoots}) we obtain $c_1=0$. For the converse, note that if $c_1=0$, then $c_2=\gamma i$ and $c_3=-\gamma i$, hence $a_2$ and $a_3$ are real.

We now consider the case where $e\neq 0$. Using the fact that $c_2c_3$ is positive and $e=-c_1c_2c_3$, it follows that$$-\sign(e)=\sign(c_1)=\sign(a).$$ Hence, a minimal polynomial for $a$ is obtained as the monic irreducible factor of $h(x)$ with a root in $(0,\infty)$ in case $e<0$, or in $(-\infty,0)$ in case $e>0$. Moreover, $a$ is the unique root of this factor in this interval. \medskip

(ii) In the case where $k(s)$ has three real roots, they must be pairwise different, otherwise $k(s)$ and its derivative $k'(s)$ have a common
root, which implies that $f$ is degenerate, since the coefficient of $xy^4$ in (\ref{eqc'}) vanishes for $c$ a double root of $k(s)$. Without loss of generality, we can assume that $c_1<c_2<c_3$. Therefore $a$ can attain the three different values
$$a_j:=\frac{3c_j}{\sqrt{|d+3c_j^2|}},\quad j=1,2,3.$$

In case $c_2=0$, we have $c_3=-c_1=\sqrt{-d}$ with $d<0$, which implies that $a_1=-\frac{3}{\sqrt{2}}$, $a_2=0$, and $a_3=\frac{3}{\sqrt{2}}$. Considering the sign of $3c_i^2+d$,  we obtain that $c_1$ and $c_3$ correspond to type $J_{10}^+$ and $c_2$ corresponds to $J_{10}^-$.

Next we consider the case $c_2\neq0$.
Since $c_1+c_2+c_3=0$, not all three roots have the same
sign. Suppose $-\sign(c_1)=\sign(c_2)=\sign(c_3)$. Note that $|c_1|=|c_2+c_3|=|c_2|+|c_3|$, that is, $|c_1|>|c_2|$ and
$|c_1|>|c_3|$. Then $3c_1^2+d=3c_1^2+c_1c_2+c_1c_3+c_2c_3>0$, since
$|c_1^2|>|c_1c_3|$ and $|c_1^2|>|c_1c_2|$, and, moreover, $c_1^2>0$, $c_1c_2<0$,
$c_1c_3<0$, and $c_2c_3>0$. Hence, $c_1$ corresponds to type $J_{10}^+$.
Furthermore
$$3c_j^2+d=3c_j^2+c_1(c_2+c_3)+c_2c_3=3c_j^2-(c_2+c_3)(c_2+c_3)+c_2c_3
=3c_j^2-(c_2^2+c_2c_3+c_3^2),$$
for $j=2,3$. Since $|c_2|<|c_3|$, we obtain that $c_2$ corresponds to type
$J_{10}^-$ and $c_3$ to $J_{10}^+$. The remaining case $-\sign(c_1)=-\sign(c_2)=\sign(c_3)$ can be treated in a similar manner, leading again to types $J_{10}^+$, $J_{10}^-$ and $J_{10}^+$ corresponding to $c_1$, $c_2$ and $c_3$.

We now determine the corresponding values of the moduli parameter. We first consider the case $J_{10}^-$. Note that $$-d(a_2)^2=c_2^2(9+3(a_2)^2)\text{\hspace{2mm}and\hspace{2mm}}-d(ia_j)^2=c_j^2(9+3(ia_j)^2),\quad j=1,3.$$ By multiplying $\tilde k(s)$ with $9+3(ia_j)^2$, for $j=1,3$, and with $9+3a_j^2$, for $j=2$, it follows that the roots of
\[h^-(z):=(-4d^3-27e^2)z^6+(-36d^3-243e^2)z^4+(-81d^3-729e^2)z^2-729e^2\in\mathbb{Q}[z]\]
are $\pm i a_1$, $\pm a_2$, and $\pm i a_3$. 
Since $\sign(c_1)=-\sign(c_3)$, we have$$\sign(a_2)=\sign(c_2)=\sign(e),$$ which determines an appropriate interval containing $a_2$.

For the parameters which correspond to $J_{10}^+$, we can construct, in a similar manner, the polynomial
 \[h^+(z):=(4d^3+27e^2)z^6+(-36d^3-243e^2)z^4+(81d^3+729e^2)z^2-729e^2\in\mathbb{Q}[z]\]
with roots $\pm a_1$, $\pm i a_2$, and $\pm a_3$.
Since, in these cases, either $|c_2|<|c_1|$ or $|c_2|<|c_3|$, it follows that $d<0$. An easy calculation shows that $|c_1|<|c_3|$ if and only if $|a_3|<|a_1|$. 

If $e<0$, that is, $c_2<0$, we have that $|c_1|+|c_2|=|c_1+c_2|=|c_3|$ and, hence, that $|a_3|<|a_1|$.  Therefore, $a_1$ is the smallest negative real root and $a_3$ is the smallest positive real root of $h^+(z)$. 
Similarly, if $e>0$, then $|a_1|<|a_3|$, hence $a_1$ is the largest negative real root, and $a_3$ the largest positive real root of $h^+(z)$.

\begin{remark}
  With regard to the implementation, we use the \textsc{Singular} library {\tt solve.lib} \citep{solvelib} and Sturm chains, as implemented in the 
  library {\tt rootsur.lib} \citep{roots}, to determine intervals with rational boundaries containing the roots.
  \end{remark}

\subsection{The \texorpdfstring{$\boldsymbol{X_9}$}{X9} case}

According to Theorem~29 in \citet{realclassify2}, the real right equivalence
class of a singularity of real main type $X_9$ always contains exactly two
normal form equations from Arnold's list, of possibly different real subtypes.
There are four different possible cases for a given singularity of real main
type $X_9$:
\begin{enumerate}[label=(\Alph*)]
\item\label{enum:X_9:caseA}
The singularity is right equivalent to $\NF\bigl(X_9^{++}\bigr)(a)$ for two
different values of the parameter~$a$ with $a > -2$.

\item\label{enum:X_9:caseB}
The singularity is right equivalent to $\NF\bigl(X_9^{--}\bigr)(a)$ for two
different values of the parameter~$a$ with $a < 2$.

\item\label{enum:X_9:caseC}
The singularity is right equivalent to both $\NF\bigl(X_9^{+-}\bigr)(a)$ and
$\NF\bigl(X_9^{-+}\bigr)(a)$ for some unique value $a \in \R$ of the parameter.

\item\label{enum:X_9:caseD}
The singularity is right equivalent to $\NF\bigl(X_9^{++}\bigr)(a)$ for some
value $a < -2$ of the parameter and to $\NF\bigl(X_9^{--}\bigr)(a)$ for some
$a > 2$.
\end{enumerate}

In the following, we discuss how Algorithm~\ref{alg:X_9} determines in which of
these cases a given singularity falls, and which are the corresponding values
of the parameter. We do not strictly follow Algorithm~\ref{alg:ClassPara}
because line~\ref{line:linearCoordinateChange} would, in general, require
working over algebraic extensions. Due to this difficulty, especially Step~(IV)
needs a specific approach.

\begin{algorithm}[htp]
\caption{Algorithm for the case $X_9$}%
\label{alg:X_9}
\begin{algorithmic}[1]

\Require{$f \in \m^3 \subset \Q[x,y]$ of complex singularity type $T = X_9$}

\Ensure{all normal form equations in the right equivalence class of $f$, each
specified as a tuple of the real singularity type, the normal form, the minimal
polynomial of the parameter, and an interval such that the parameter is the
unique root of the minimal polynomial in this interval}

\State $f := \jet(f, 4)$
\State\label{line:X_9:linear_trafo}%
  make sure that the coefficient of $x^4$ in $f$ is non-zero by applying either
  $(x, y) \mapsto (y, x)$, $(x, y) \mapsto (x, x+y)$,
  $(x, y) \mapsto (x, 2x+y)$, or $(x, y) \mapsto (x, 3x+y)$ to $f$ if necessary
\State\label{line:X_9:Tschirnhaus}%
  apply $x \mapsto x-\frac{\coeff(f, x^3y)}{\coeff(f, x^4)}y,\; y \mapsto y$ to
  $f$
\State write $f$ as $f = bx^4+cx^2y^2+dxy^3+ey^4$ with $b, c, d, e \in \Q$ and
  $b \neq 0$

\State\label{line:X_9:minpoly}%
$\begin{aligned}[t]
p^{(\sigma)} :={}
& \left(-256b^3e^3+128b^2c^2e^2-144b^2cd^2e+27b^2d^4-16bc^4e+4bc^3d^2\right)
  \cdot z^6 \\
& +\sigma \left(18432b^3e^3+11520b^2c^2e^2-5184b^2cd^2e+972b^2d^4+144bc^3d^2
  +16c^6\right) \cdot z^4 \\
& +\left(-331776b^3e^3-62208b^2cd^2e+11664b^2d^4-11520bc^4e+1728bc^3d^2
  -128c^6\right) \cdot z^2 \\
& +\sigma \left(331776b^2c^2e^2-248832b^2cd^2e+46656b^2d^4-18432bc^4e
  +6912bc^3d^2+256c^6\right) \\
& \in \Q[z]
\end{aligned}$

\State let $r$ be the number of real roots of $f(x, 1) = bx^4+cx^2+dx+e$

\If{$r = 0$}
  \If{$b > 0$}
    \If{$f \sim +x^4+ax^2y^2+y^4$ for some $a \in [0, 2)$}%
        \label{line:X_9:caseAstart}
      \State $I_1 := [0, 2)$, $I_2 := (2, 6]$
    \Else
      \State $I_1 := (-2, 0)$, $I_2 := (6, \infty)$
    \EndIf
    \State $C_1 := \left( X_9^{++},\; +x^4+ax^2y^2+y^4,\; m_{I_1}(p^+),\; I_1
      \right)$
    \State $C_2 := \left( X_9^{++},\; +x^4+ax^2y^2+y^4,\; m_{I_2}(p^+),\; I_2
      \right)$%
        \label{line:X_9:caseAend}
  \Else
    \If{$f \sim -x^4+ax^2y^2-y^4$ for some $a \in (-2, 0]$}
      \State $I_1 := [-6, -2)$, $I_2 := (-2, 0]$
    \Else
      \State $I_1 := (-\infty, -6)$, $I_2 := (0, 2)$
    \EndIf
    \State $C_1 := \left( X_9^{--},\; -x^4+ax^2y^2-y^4,\; m_{I_1}(p^+),\; I_1
      \right)$
    \State $C_2 := \left( X_9^{--},\; -x^4+ax^2y^2-y^4,\; m_{I_2}(p^+),\; I_2
      \right)$
  \EndIf
\EndIf

\If{$r = 2$}
  \If{$f \sim +x^4+ax^2y^2-y^4$ for some $a \in (-\infty, 0]$}
    \State $I := (-\infty, 0]$
  \Else
    \State $I := (0, \infty)$
  \EndIf
  \State $C_1 := \left( X_9^{+-},\; +x^4+ax^2y^2-y^4,\; m_I(p^-),\; I \right)$
  \State $C_2 := \left( X_9^{-+},\; -x^4+ax^2y^2+y^4,\; m_I(p^-),\; I \right)$
\EndIf

\If{$r = 4$}
  \If{$f \sim +x^4+ax^2y^2+y^4$ for some $a \in (-\infty, -6]$}
    \State $I_1 := (-\infty, -6]$, $I_2 = (2, 6]$
  \Else
    \State $I_1 := (-6, -2)$, $I_2 = (6, \infty)$
  \EndIf
  \State $C_1 := \left( X_9^{++},\; +x^4+ax^2y^2+y^4,\; m_{I_1}(p^+),\; I_1
    \right)$
  \State $C_2 := \left( X_9^{--},\; -x^4+ax^2y^2-y^4,\; m_{I_2}(p^+),\; I_2
    \right)$
\EndIf

\Return $C_1,\, C_2$

\end{algorithmic}
\end{algorithm}

First of all, the determinacy of a given polynomial $f \in \Q[x, y]$ of complex
singularity type $X_9$ is $4$, thus it suffices to consider $\jet(f, 4)$. Some
calculations in linear algebra show that we can assume the coefficient of $x^4$
to be non-zero, by applying a linear coordinate transformation as in
line~\ref{line:X_9:linear_trafo} of Algorithm~\ref{alg:X_9} if necessary. For
convenience, we can then get rid of the term $x^3y$ by the transformation given
in line~\ref{line:X_9:Tschirnhaus} such that $f$ is of the form
\[
f = bx^4+cx^2y^2+dxy^3+ey^4 \text{ with } b, c, d, e \in \Q \text{ and }
b \neq 0.
\]

The number of real roots of $f(x, 1) = bx^4+cx^2+dx+e$, which geometrically
correspond to the points of the strict transform on the exceptional divisor of
the blow-up at the origin, is invariant under right equivalence and can thus be
used as a first step to distinguish the four cases mentioned above. In fact, in
the cases~\ref{enum:X_9:caseA} and \ref{enum:X_9:caseB} the polynomial
$f(x, 1)$ has no real roots, in the case~\ref{enum:X_9:caseC} it has two real
roots, and in the case~\ref{enum:X_9:caseD} it has four real roots.
Furthermore, the two cases~\ref{enum:X_9:caseA} and \ref{enum:X_9:caseB} can be
distinguished by the sign of the coefficient of $x^4$ because it is invariant
under right equivalence if $f$ is of either one of these cases.

It remains to determine the possible values of the parameter $a$. Using the
techniques from Section~5.1 in \citet{realclassify2}, one can determine, for
each real subtype, a polynomial whose coefficients depend on those of $f$ and
whose roots are precisely the possible values of the parameter in the normal
form equations to which $f$ is complex right equivalent. In
Algorithm~\ref{alg:X_9}, this is the polynomial $p^{(\sigma)}$ defined in
line~\ref{line:X_9:minpoly}, with $\sigma = +1$ for the subtypes $X_9^{++}$ and
$X_9^{--}$, and with $\sigma = -1$ for $X_9^{+-}$ and $X_9^{-+}$.

However, it turns out that $p^{(\sigma)}$ does not always factorize into linear
factors over $\Q$ and that the number of its real roots within the intervals
specified in the cases~\ref{enum:X_9:caseA} to \ref{enum:X_9:caseD} above is
larger than the number of admissible values for the parameter. This is due to
the fact that $p^{(\sigma)}$ also takes into account complex transformations.
In other words, for some of the real roots of $p^{(\sigma)}$, there is no real
transformation which takes $f$ to the respective normal form equation where the
value of the parameter is that root.

We use the following method to solve this problem. Considering, again,
Theorem~29 in \citet{realclassify2}, one may observe that the real roots of
$p^{(\sigma)}$ lie in fixed disjoint intervals. In the
case~\ref{enum:X_9:caseA}, for example, the polynomial $p^+$ has exactly one
real root in each of the intervals $(-2, 0)$, $(0, 2)$, $(2, 6)$, and
$(6, \infty)$ (and two more in $(-\infty, -6)$ and $(-6, -2)$, which are not
admissible values for the parameter in this case), or it has two double roots
at $0$ and $6$ (and one more double root at $-6$, which we do not consider
either). According to that theorem, the two roots which are admissible values
for the parameter are either those in $[0, 2)$ and $(2, 6]$ or those in
$(-2, 0)$ and $(6, \infty)$.

Using the techniques from Section~5.3 in \citet{realclassify2}, we set up, for
a generic parameter $a$, the ideal of transformations which take $f$ to the
normal form equation with this parameter. Note that a real point in the
vanishing set of this ideal corresponds to a real transformation of $f$. We can
thus determine whether there exists a real transformation which maps $f$ to a
normal form equation where the parameter lies in a specific interval. To determine this, we can use an algorithm for real root isolation implemented in the \textsc{Singular} library \texttt{rootisolation.lib} \citep{rootisolationlib}. Applying ideas from \citet[Ch. 6]{SW}, it is based on a subdivision method using interval arithmetic to do an exclusion test and an interval Newton step for an inclusion test. Alternatively one could also use methods based on quantor ellimination as described in Algorithm~12.8 from \citet{BPR}, however, these methods are less efficient.  
Continuing with
case~\ref{enum:X_9:caseA}, if there exists a real transformation for the real
root of $p^+$ in $[0, 2)$, then the two admissible values of the parameter are
given by this root and the one in $(2, 6]$, otherwise they are given by the two
roots of $p^+$ in $(-2, 0)$ and $(6, \infty)$, see
lines~\ref{line:X_9:caseAstart} - \ref{line:X_9:caseAend} in
Algorithm~\ref{alg:X_9}. The cases~\ref{enum:X_9:caseB}, \ref{enum:X_9:caseC},
and \ref{enum:X_9:caseD} are treated in an analogous way.

\begin{remark}
Instead of determining the applicable case and then use a prior known intervals, one can also use an ansatz and  real root isolation to find appropriate intervals for the possible parameter values along with the computation of the case. To test whether $f = bx^4+cx^2y^2+dxy^3+ey^4$ with $b, c, d, e \in \Q$ is right equivalent to $f_0=\sigma_1 x^4 +ax^2y^2+\sigma_2 y^4$ with fixed signs $\sigma_i$ and undetermined $a$, we make an ansatz for the  right equivalence transformation $\phi(x)=\alpha x +\beta y$, $\phi(y)=\gamma x+ \delta y$ and isolate the roots $(\alpha,\beta,\gamma,\delta, a)\in \mathbb{R}^5$ for the polynomial system arising from the condition $\phi(f)=f_0$. This system is represented by an ideal $I\subset \mathbb{Q}[\alpha,\beta,\gamma,\delta, a]$. If this system has a real solution, then $f$ is of type $X_9^{\sigma_1\sigma_2}$. 

Root isolation applied to $I$ yields Cartesian products of intervals, where each product contains a unique solution. Let $p$ be a generator of the elimination ideal $ I \cap \mathbb{Q}[a]$. We can design the root isolation algorithm in a way that each of the $a$-intervals of the solutions of $I$ contains precisely one solution of $p$. Then the minimal non-empty intersections of 
$a$-intervals of solutions of $I$ 
 lead to intervals isolating the possible parameter values as solutions of $p$. A minimal polynomial for the parameter value is obtained as an irreducible factor of $p$. Recall that it is not sufficient to do root isolation on $p$, since the transformations $\phi$ corresponding to these roots may not be defined over the reals.

For performance reasons, we can consider a projection $\mathbb{R}^5\rightarrow \mathbb{R}^2$, $(\alpha,\beta,\gamma,\delta, a)\mapsto (c_1\alpha+c_2\beta+c_3\gamma+c_4\delta , a)$ with $c_i\in \mathbb{Q}$. 
In every case of Algorithm \ref{alg:X_9} the number values for $a$ occurring for real solutions is known. If the number of solutions computed from the projection coincides with the expected number, the projection was general enough, otherwise we choose a different projection.
\end{remark}

\section{Hyperbolic Singularities}\label{sec:Hyper}

In this section, we give details on the application of the general algorithm (Algorithm \ref{alg:ClassPara}) for the hyperbolic singularities. See Figure \ref{fig hyperbolic} for these cases.

\begin{figure}[ht]
\begin{center}
\setlength{\tabcolsep}{3mm}
\begin{tabular}
[c]{ccc}%
{\includegraphics[
height=2.3in,
]%
{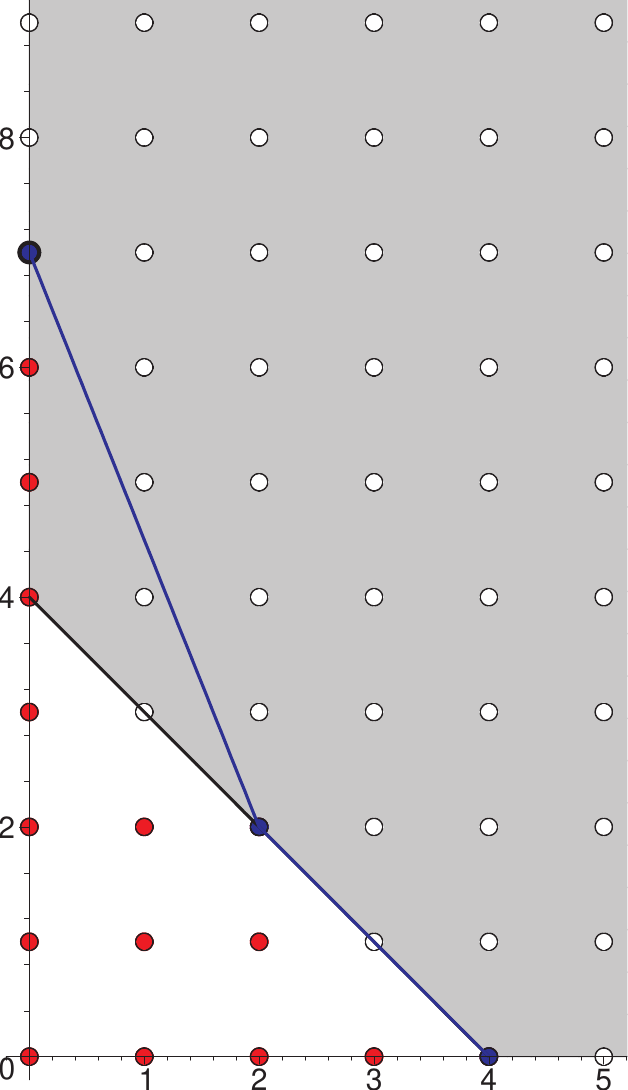}%
}%
&
{\includegraphics[
height=2.3in,
]%
{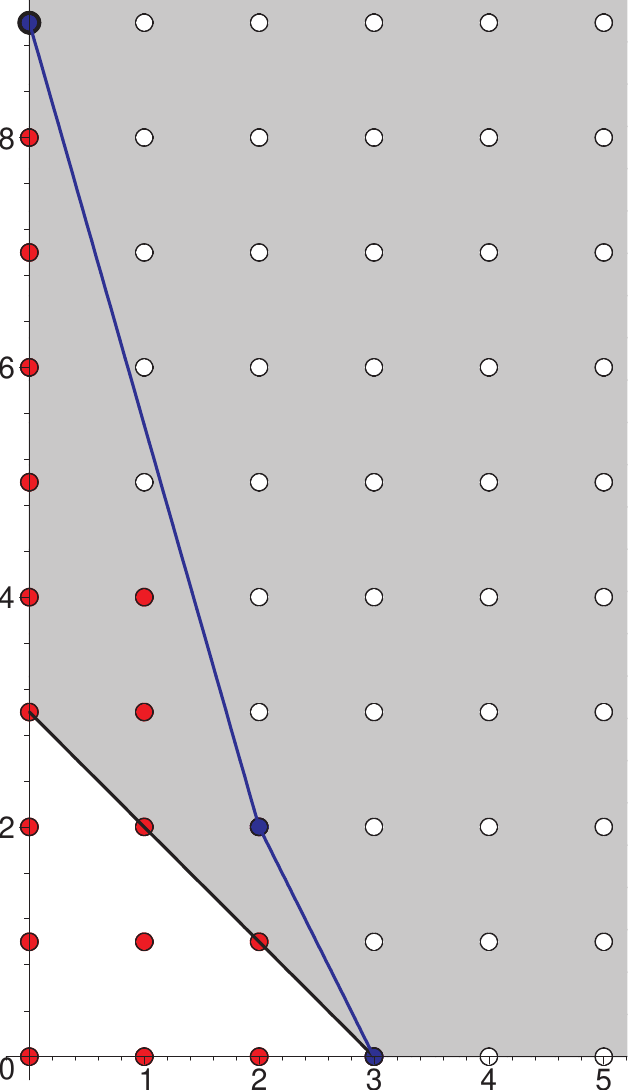}%
}%
&
{\includegraphics[
height=2.3in,
]%
{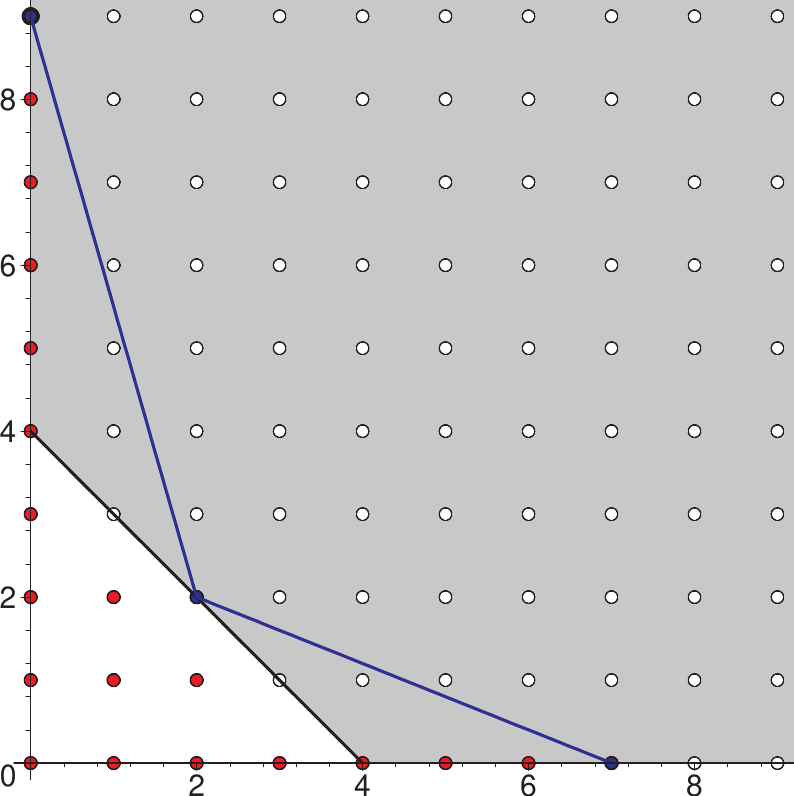}%
}%
\\
$X_{9+k}$ & $J_{10+k}$ & $Y_{r,s}$ and $ \tilde{Y}_{r}$
\end{tabular}
\caption{Hyperbolic Singularities}%
\label{fig hyperbolic}%
\end{center}
\end{figure}

\subsection{The $X_{9+k}$ case.}

In Algorithm \ref{alg:X_{9+k}} we demonstrate the general algorithm in the cases $X_{9+k}$. Note that the Milnor number $\mu$ of the input polynomial $f$ can be easily computed, using Gr\"obner basis techniques, and that $k=\mu-9$, see \cite{AVG1985}. 
Also note that $
E_5\cap E_{2\mu-9}^{(\mu-7,2)}$ is
the vector space generated by all monomials above the Newton polygon. 
    In line \ref{line X9k odd} we use that if the exponent of the $y^{k+4}=y^{\mu-5}$ term is odd, there are two possible normal form equations, which can be obtained from each other by the transformation $x\mapsto x$, $y\mapsto -y$.
In the algorithm we use the following notation.
\begin{notation}
If $I\subset \R$ and $\lambda\in\R$, then we define $\lambda I$ as the set 
\[\lambda I:=\{\lambda x\mid x\in I\}.\]
\end{notation}
The cases $J_{10+k}$ can be handled in a similar manner.

\begin{algorithm}[ht]
\caption{Algorithm for the case $X_{9+k}$}%
\label{alg:X_{9+k}}
\begin{algorithmic}[1]

\Require{$f \in \m^3\subset\Q[x,y]$ of complex singularity type $T=X_{9+k}$}

\Ensure{
all normal form equations in the right equivalence class of $f$, each specified as a tuple of the real singularity type, the normal form, the minimal polynomial of the parameter, and an interval such that the parameter is the unique root of the minimal polynomial in this interval
}

\vspace{1mm}

\noindent\textit{II. Eliminate the monomials in $\supp(f)$ underneath or on $\Gamma(T)$}

\vspace{1mm}

\State apply Algorithm \ref{alg:1jetX9+k} to $f$

\vspace{0.1cm}
\noindent\textit{Use Algorithm \ref{alg:Transformation} iteratively:}
\vspace{0.1cm}

\State $\mu:=9+k$
\State $c := \coeff(f,x^2y^2)$

\For{$i=4\ldots,\left\lfloor\frac{\mu-3}{2}\right\rfloor$}
\State $t :=  \coeff(f,xy^{i})$
\State apply $x\mapsto x-\frac{t}{2c}y^{i-2}$, $y\mapsto y$ to $f$
\EndFor

\vspace{0.1cm}
\noindent\emph{III. and IV. Discard the monomials above $\Gamma(T)$ and read off the desired information:}
\vspace{0.1cm}

\State write $f$ as $f = b_0x^4+b_1x^2y^2+b_2y^{\mu-5}$ with $b_0,b_1,b_2\in\Q^{*}$ and $R\in E_5\cap E_{2\mu-9}^{(\mu-7,2)}$\label{line X9k III and IV}
\vspace{1mm}
\State $T_{\R}:= X_{9+k}^{\operatorname{sign}(b_0),\ \operatorname{sign}(b_1)}$\vspace{1mm}
\State $F:= \operatorname{sign}(b_0)\cdot x^4+ \operatorname{sign}(b_1)\cdot x^2y^2+ay^{\mu-5}$\vspace{1mm}
\State $p:= z^4-b_2^4\left(\frac{|b_0|}{b_1^{2}}\right)^{\mu-5}$\vspace{1mm}

\State $I_1:=\operatorname{sign}(b_2)\cdot(0,\infty)$, $I_2:=\operatorname{sign}(b_2)\cdot(-\infty,0)$

\If{$\mu$ is odd}\label{line X9k odd}

\Return $(T_{\R},\ F,\ \operatorname{m}_{I_1}(p),\ I_1)$ 

\Else

\Return $(T_{\R},\ F,\ \operatorname{m}_{I_1}(p),\ I_1),\ (T_{\R},\ F,\ \operatorname{m}_{I_2}(p),\ I_2)$

\EndIf

\end{algorithmic}
\end{algorithm}

\subsection{The \texorpdfstring{$\boldsymbol{Y_{r,s}}$}{Yrs} case}%
\label{sec Y}

By the following lemma, if the given polynomial $f\in\Q[x,y]$ is of real main type $Y_{r,s}$ with $r\neq s$, then the $4$-jet of $f$  always factors into linear factors over $\Q$. Hence we can follow the general algorithm without any additions. The same may happen if $f$ is of real main type $Y_{r,r}$. However, if $f$ is of real main type $Y_{r,r}$ and the $4$-jet of $f$ does not factor into linear factors over $\Q$, line \ref{line:linearCoordinateChange} of Algorithm \ref{alg:ClassPara} requires a real algebraic field extension.

\begin{lemma}\label{IrrationalPoints}
Let $f\in\Q[x,y]$ be of real main type $Y_{r,s}$. If the $4$-jet of $f$ does not factorize into linear factors over $\Q$, then $r=s$, that is, $f$ is of real main type $Y_{r,r}$.
\end{lemma}

\begin{proof}
Let $f$ be of real main type $Y_{r,s}$ for some $r,s>4$. 
Then $f$ is of the form
\[f=(a_0x+a_1y)^2(b_0x+b_1y)^2+\textnormal{ higher terms in $x$ and
$y$}\] where $(a_0,a_1), (b_0,b_1)\in\R^2$ are linearly independent. 
Hence, the strict transform of the blow-up of $f$ at the origin intersects the exceptional divisor in exactly two points. 

Without loss of generality, we consider
the chart where $x=0$ is the local equation of the exceptional divisor. Then the strict transform is given by
\[\tilde f=(a_0+a_1y)^2(b_0+b_1y)^2+\textnormal{ terms that are divisible by
$x$.}\] Therefore, the intersection points of the strict transform with the exceptional divisor $\{x=0\}$ correspond to the zeros of the rational polynomial $p^2 =
(a_0+a_1y)^2(b_0+b_1y)^2$.  Since $\mathbb Q$ is a
perfect field and $p^2\in\Q[y]$, it
follows that $p=(a_0+a_1y)(b_0+b_1y)\in\Q[y]$.

Now consider the standard charts where $x=0$ and $y=0$, respectively, define the exceptional divisor. We show that if the strict transform has an irrational point on the exceptional divisor $E$ in any of the two charts, then in both charts it has two irrational points on $E$. Indeed, if $a_1=0$ or $b_1=0$, then $p$ has only rational roots. Hence, by the assumption that the strict transform has an irrational point on $E$, we have
$a_1\neq 0$ and $b_1\neq 0$, and at least one of the two roots $-\frac{a_0}{a_1}$, $-\frac{b_0}{b_1}$ is irrational.
Since $p$ is the minimal polynomial for these roots, we conclude that both $-\frac{a_0}{a_1}$ and $-\frac{b_0}{b_1}$ are irrational.

Note that these roots correspond to the irrational points $(0,-\frac{a_0}{a_1})$ and $(0,-\frac{b_0}{b_1})$ of the strict transform on $E$. Hence, the points of the strict transform on the exceptional divisor are irrational if and only if $p$ does not have any rational root. This, in turn, is the case if and only if $\jet(f,4)=(a_0x+a_1y)^2(b_0x+b_1y)^2$ does not factorize into linear factors over $\Q$.

Thus, under the assumptions of the lemma, both roots of $p$ and both points of the strict transform on the exceptional divisor are irrational. 
Consider the blow-up of the normal form equation of $f$ at the origin.
The germs at the two intersection points $q_1$ and $q_2$ of the strict transform with the exceptional divisor are right equivalent to $f_1=\pm x^2\pm y^{r-4}$
and  $f_2=\pm x^2\pm y^{s-4}$. Since any right equivalence of a singularity induces right equivalences for each germ of the strict transform, 
the singularities of the strict transform of $f$ are also right equivalent to $f_1$ and~$f_2$.

The Galois group of the quadratic extension $\Q \subset \Q[y]/p =: K$ is
isomorphic to $\mathbb Z/2\mathbb Z$. Let $\kappa$ be the non-trivial element
in this group, and let $\varphi_1$ and $\varphi_2$ be the $K$-algebra
automorphisms which transform the germs at $q_1$ and $q_2$ to $f_1$ and $f_2$,
respectively.
Then $q_1$ and $q_2$ are conjugate via $\kappa$, and we have
$\kappa^{-1} \circ \varphi_1 \circ \kappa = \kappa(\varphi_1) = \varphi_2$ as
$K$-algebra automorphisms because both $\widetilde{f}$ and $f_1$ are invariant
under $\kappa = \kappa^{-1}$. In other words, $\varphi_1$ and $\varphi_2$ can be
identified with each other by the group action of $\operatorname{Gal}(K|\Q)$ on
$\operatorname{Aut}(K)$.
Also note that the map $\kappa^{-1} \circ \varphi_1 \circ \kappa$ can be extended
to an $\mathbb{R}$-algebra automorphism because its restriction to $K$ is just
the identity, and note that the germ at $q_2$ is right equivalent to $f_1$ via this
map. Hence $f_1 \sim f_2$ and thus $r = s$.

\end{proof}

We now discuss an explicit version of Algorithm \ref{alg:ClassPara} which works for all rational input polynomials $f\in\Q[x,y]$ of complex type $Y_{r,s}$. The framework is given in Algorithm \ref{alg:Y_{r,s}}. We realize Step (II) of the general algorithm in Algorithm \ref{alg:Y_{r,s}2}  and Step (IV) in Algorithm \ref{alg:Y_{r,s}IV}.

\begin{algorithm}[htp]
\caption{Algorithm for the cases $Y_{r,s}$
}%
\label{alg:Y_{r,s}}
\begin{algorithmic}[1]

\Require{$f \in \m^3\subset\Q[x,y]$ of complex singularity type $T=Y_{r,s}$}

\Ensure{
all normal form equations in the right equivalence class of $f$, each specified as a tuple of the real singularity type, the normal form, the minimal polynomial of the parameter, and an interval such that the parameter is the unique root of the minimal polynomial in this interval
}

\State $f_0:=f$

\noindent\textit{II. Eliminate the monomials in $\supp(f)$ underneath or on $\Gamma(T)$}

\vspace{1mm}

\State replace $f$ by the output of Algorithm \ref{alg:Y_{r,s}2} applied to $f$\vspace{1mm}

\noindent\textit{III. Eliminate the monomials above $\Gamma(T)$}
\vspace{1mm}

\State $f := w\dash\jet(f,d_w)$ with $w:=w(T)$ and $d_w:=d(T)$ 
\vspace{1mm}

\noindent\textit{IV. Read off the desired information}
\vspace{1mm}
\State Output of Algorithm \ref{alg:Y_{r,s}IV} applied to $f_0$ and $f$
\end{algorithmic}
\end{algorithm}

\begin{algorithm}[htp]
\caption{Step (II) in the general algorithm for the cases $Y_{r,s}$ 
}%
\label{alg:Y_{r,s}2}
\begin{algorithmic}[1]
\Require{$f \in \m^3\subset\Q[x,y]$ of complex singularity type $T=Y_{r,s}$}
\Ensure{$f'\in \m^3\subset K[x,y]$, where $K$ is $\Q$ or a quadratic extension field of $\Q$, such that $f\overset{K}{\sim} f'$ and $f'$ has no terms underneath $\Gamma(T)$ }

\State $h:=\jet(f,4)$

\State $b:=\coeff(f,x^2y^2)$

\noindent\textit{Implement line \ref{line:linearCoordinateChange} from Algorithm \ref{alg:ClassPara}:}

\If{$h$ has a linear factor over $\Q$}\label{line:case1(a)}

\State factorize $h$ as $bg_1^2g_2^2$ over $\Q$, where $g_1,g_2$ are homogeneous of degree $1$
\State\label{line:case1(b)} apply $g_1\mapsto x$, $g_2\mapsto y$ to $f$
\Else\label{line:case2(a)}

\State factorize $h$ as $bg^2$ over $\Q$, where $g$ is homogeneous of degree $2$
\State $K:= \Q[t]/g(1,t)$
\State over $K$, factorize $h$ as $bg_1^2g_2^2$, where $g_1,g_2$ are homogeneous of degree $1$
\State\label{line:case2(b)} apply $g_1\mapsto x$, $g_2\mapsto y$ to $f$
\EndIf

\noindent\textit{Implement lines \ref{line:n2forloop} - \ref{line:TwoFacesEnd0} from Algorithm \ref{alg:ClassPara}:}

\State \label{line:case1case2(a)}$n_1:=3$, $n_2:=3$
\While{ $\coeff(f,x^{n_1})=0$ or $\coeff(f,y^{n_2})=0$}
\If {$\coeff(f,x^{n_1})=0$}
\State $n_1:= n_1+1$
\State apply $x\mapsto x$ , $y\mapsto y-\frac{\coeff(f,x^{n_1}y)}{2b}x^{n_1-2}$ to $f$
\EndIf
\If{$\coeff(f,y^{n_2})=0$}
\State  $n_2:=n_2+1$
\State\label{line:case1case2(b)} apply $x\mapsto x-\frac{\coeff(f,xy^{n_2})}{2b}y^{n_2-2}$, $y\mapsto y$ to $f$
\EndIf
\EndWhile 
\Return $f':=f$
\end{algorithmic}
\end{algorithm}

Suppose $f$ is of real main type $Y_{r,s}$. Algorithm \ref{alg:Y_{r,s}2} removes all monomials below or on the Newton polygon, with the result defined either over $\Q$ or over a simple real algebraic extension of $\Q$. From this, Algorithm~\ref{alg:Y_{r,s}IV} computes a linear or quadratic minimal polynomial for the moduli parameter. The algorithm then determines which of the roots of the minimal polynomial are valid moduli parameters and computes the corresponding real subtypes.

We now give a detailed exposition of Algorithms \ref{alg:Y_{r,s}2} and \ref{alg:Y_{r,s}IV}. Before entering the algorithm we remember $f_0=f$. For Algorithm \ref{alg:Y_{r,s}2} we set $h:=\jet(f,4)$. The algorithm then considers the following two cases:

\begin{algorithm}[p]
\caption{Step (IV) in the general algorithm for the real main types $Y_{r,s}$}%
\label{alg:Y_{r,s}IV}
\begin{algorithmic}[1]
\Require{$f_0 \in \m^3\subset\Q[x,y]$ of real singularity type $ Y_{r,s}$ and $f\in\m^3\subset K[x,y]$, where $K=\mathbb{Q}$ or $K$ is a real quadratic extension of $\mathbb{Q}$, such that $f$ is of complex singularity type $T=Y_{r,s}$, $\supp(f)= \supp(T)$, and $f\overset{K}{\sim} f_0$}
\Ensure{
all normal form equations in the right equivalence class of $f_0$, each specified as a tuple of the real singularity type, the normal form, the minimal polynomial of the parameter, and an interval such that the parameter is the unique root of the minimal polynomial in this interval
}
\State $h:=\jet(f_0,4)$
 \If{$h$ has a linear factor over $\Q$}\label{line:caseIV1(a)}
\State write $f$ as $f=bx^2y^2+dx^{r}+ey^{s}$ with $b,d,e\in \mathbb{Q}^{*}$
\If{$r$ is odd and $s$ is even}
\State apply $x\mapsto y$, $y\mapsto x$ to $f$
\State write $f$ as $f=bx^2y^2+dx^{r}+ey^{s}$  with $b,d,e\in \mathbb{Q}^{*}$
\EndIf
\State $p:=z^{2r}-|b|^{-rs}d^{2s}e^{2r}$,\hspace{1mm} $\tilde p:=z^{2s}-|b|^{-rs}d^{2s}e^{2r} \in \mathbb{Q}[z]$
\State $\sigma_1 := \sign(b)$, $\sigma_2 := \sign(d)$, $\sigma_3 := \sign(e)$, $I :=(0,\infty)$
\If{$r$ and $s$ are even}
\State $C_1 := (Y_{r,s}^{\sigma_1\sigma_2},\  \sigma_1x^2y^2+\sigma_2x^{r}+ay^{s},\ \operatorname{m}_{\sigma_3\cdot I}(p),\  \sigma_3\cdot I )$
\State $C_2 := (Y_{s,r}^{\sigma_1\sigma_3},\  \sigma_1x^2y^2+\sigma_3 x^{s}+ay^{r},\ \operatorname{m}_{\sigma_2\cdot I}(\tilde p),\  \sigma_2\cdot I )$
\If {$r= s$ and $\sigma_2=\sigma_3$}
   \Return $C_1$
\Else
   \Return $C_1$, $C_2$
\EndIf

\EndIf
\If{$r$ is even and $s$ is odd}
\For{$i=1,2$}
\State $C_i := (Y_{r,s}^{\sigma_1\sigma_2},\ \sigma_1x^2y^2+\sigma_2x^{r}+ay^{s},\ \operatorname{m}_{(-1)^i\cdot I}(p),\ (-1)^i\cdot I )$
\State $C_{i+2} := (Y_{s,r}^{\sigma_1\, (-1)^i},\  \sigma_1x^2y^2+(-1)^i x^{s}+ay^{r},\ \operatorname{m}_{\sigma_2\cdot I}(\tilde p),\ \sigma_2\cdot I )$
\EndFor
\Return $C_1$,\ $C_2$,\ $C_3$,\ $C_4$
\EndIf
\If{$r$ and $s$ are odd}
\For {$i=1,2$}
\For{$j=1,2$}
\State $C_{2i+j-2}:=(Y_{r,s}^{\sigma_1\, (-1)^i},\  \sigma_1x^2y^2+(-1)^ix^{r}+ay^{s},\ \operatorname{m}_{(-1)^j\cdot I}(p),\ (-1)^j\cdot I )$
\State $C_{2i+j+2} := (Y_{s,r}^{\sigma_1\, (-1)^i},\  \sigma_1x^2y^2+(-1)^ix^{s}+ay^{r},\ \operatorname{m}_{(-1)^j\cdot I}(\tilde p),\ (-1)^j\cdot I )$
\EndFor
\EndFor
\If {$r= s$}
\Return $C_1$, $C_2$, $C_3$, $C_4$
\Else
\Return $C_1$, $C_2$, $C_3$, $C_4$, $C_5$, $C_6$, $C_7$, $C_8$ 
\EndIf
\EndIf\label{line:caseIV1(b)}
\Else\label{line:caseIV2(a)}
\State write $f$ as $f=bx^2y^2+dx^{r}+ey^{r}$  with $b\in \mathbb{Q}^{*}$ and $d,e\in K^{*}$
\State $\sigma :=\sign(b)$, $I:=(0,\infty)$
\If{$r$ is odd}\label{line:caseIV2(a)2}
\State $p:= z^2-|b|^{-r}(de)^2\in\mathbb{Q}[z]$
\For{$i=1,2$}
\State $C_i := (Y_{r,r}^{\sigma+},\ \sigma x^2y^2+x^{r}+ay^{r},\ \operatorname{m}_{(-1)^i\cdot I}(p),\ (-1)^i\cdot I)$
\State $C_{i+2} := (Y_{r,r}^{\sigma-},\ \sigma x^2y^2-x^{r}+ay^{r}, \operatorname{m}_{(-1)^i\cdot I}(p),\ (-1)^i\cdot I)$
\EndFor \label{line:caseIV2m}
\Return $C_1$, $C_2$, $C_3$, $C_4$

\algstore{breakalg}
\end{algorithmic}
\end{algorithm}

\begin{algorithm}[h]
\begin{algorithmic}[1]
\algrestore{breakalg}

\Else\label{line:caseIV2<0}
\If{$de<0$}
\For{$i=1,2$}
\State $C_i:=(Y^{\sigma\, (-1)^i}_{r,r},\ \sigma x^2y^2+(-1)^ix^{r}+ay^{r},\ z-(-1)^i|b|^{\frac{-r}{2}}de,\ (-\infty,\infty))$
\EndFor
\Return $C_1$, $C_2$ \label{line:caseIV2<02}
\Else \label{line:caseIV2m2}
\State write $g(1,t)$ as defined in Algorithm \ref{alg:Y_{r,s}2} as $g(1,t)=\alpha t^2+\beta t +\gamma$ and $d=\lambda_1+\lambda_2t$
\If{$\sign(\lambda_1-\lambda_2\cdot \frac{\beta}{2\alpha})=-1$ }
\Return $(Y^{\sigma\, -}_{r,r},\ \sigma x^2y^2-x^{r}+ay^{r},\ z+|b|^{\frac{-r}{2}}de,\ (-\infty,\infty))$
\Else
\Return $(Y^{\sigma\, +}_{r,r},\ \sigma x^2y^2+ x^{r}+ay^{r},\ z-|b|^{\frac{-r}{2}}de,\ (-\infty,\infty))$
\EndIf
\EndIf\label{line:caseIV2(b)}
\EndIf
\EndIf
\end{algorithmic}
\end{algorithm}

\begin{itemize}[leftmargin=10mm]
\item[1.] $h(1,y)$ has a linear factor over $\Q$:

\vspace{0.1cm}

\noindent This case is handled in lines \ref{line:case1(a)} - \ref{line:case1(b)} and \ref{line:case1case2(a)} - \ref{line:case1case2(b)} of Algorithm \ref{alg:Y_{r,s}2}. It includes all singularities of real main type $Y_{r,s}$ with $r\neq s$ and those of type $Y_{r,r}$ for which the $4$-jet factorizes into linear factors over $\Q$. We exactly follow the general algorithm.

\vspace{0.1cm}

\item[2.] $h(1,y)$ does not have a linear factor over $\Q$: 

\vspace{0.1cm}

\noindent This case is handled in lines \ref{line:case2(a)} - \ref{line:case1case2(b)} of Algorithm \ref{alg:Y_{r,s}2}. It includes all singularities of type $Y_{r,r}$ for which the $4$-jet does not factorize into linear factors over $\Q$.  We first discuss lines \ref{line:case2(a)}  - \ref{line:case2(b)}.  Over the reals, $\jet(f,4)$ factorizes as $g_1^2g_2^2$, with $g_1$ and $g_2$ non-associate and homogeneous of degree $1$. Since $\Q$ is a perfect field, we have $g_1 g_2\in\Q[x,y]$.  Since $g_1,g_2\not\in\Q[x,y]$, the polynomial $g_1g_2(1,y)$ is irreducible. Passing to the extension field $K=\Q[y]/g_1g_2(1,y)$, we apply the automorphism $g_1\mapsto x$, $g_2\mapsto y$ to $f$.  
{Lines \ref{line:case2(b)} - \ref{line:case1case2(b)}} of Algorithm \ref{alg:Y_{r,s}2} follow the general algorithm, using transformations over $K$ to find a polynomial which is right equivalent to $f$ and of the form
\[
bx^2y^2+dx^r+ey^s+R \text{ with } b \in \Q,\, d, e \in K, \text{ and }
R \in E_{n_1+1}^{(n_1-2,2)} \cap E_{n_1+1}^{(2,n_1-2)} \,.
\]
\end{itemize}

 We now pass to Algorithm \ref{alg:Y_{r,s}IV}. Here we set $h:=\jet(f_0,4)$. 
 
 The algorithm considers the following two cases:

\begin{itemize}[leftmargin=10mm]
\item[1.]  $h(1,y)$ has a linear factor over $\Q$:

\vspace{0.1cm}

\noindent This case is handled in lines \ref{line:caseIV1(a)} - \ref{line:caseIV1(b)} of Algorithm \ref{alg:Y_{r,s}IV}, which follow the general algorithm.
We use Theorem 33 in \cite{realclassify2} to determine the possible
normal form equations, depending on whether $r$ and $s$ are even or odd, respectively.

\vspace{0.1cm}

\item[2.] $h(1,y)$ does not have a linear factor over $\Q$: 

\vspace{0.1cm}

\noindent This case is handled in lines \ref{line:caseIV2(a)} - \ref{line:caseIV2(b)} of Algorithm \ref{alg:Y_{r,s}IV}. Note that the input polynomial $f$ is an element of $K[x,y]$, where $K=\mathbb Q[t]$ is a real quadratic number field. After normalizing and completing the square in $g(1,t)$ as defined in Algorithm \ref{alg:Y_{r,s}2}, we may assume that $t=\sqrt D$ for some positive discriminant $D\in\mathbb Q$.  Furthermore, note that $f$ is of the form
\[f=bx^2y^2+dx^r+ey^r \text{ with }  b\in\mathbb Q\text{ and } d,e\in K,\]
 and let $\sigma=\sign(b)$. 
 
The case where $r$ is odd is handled in lines \ref{line:caseIV2(a)2} - \ref{line:caseIV2m}. In this case, $f$ is right equivalent to both the types $Y_{r,r}^{\sigma\pm}$, taking into account the transformation $x\mapsto -x$ and $y\mapsto y$.  By Theorem 33 from \cite{realclassify2} there are only two possible values for the parameter, which are both roots of the polynomial $z^2-|b|^{-r}(de)^2$.  If $a$ is a possible value for the parameter, then so is $-a$ by the transformation $x\mapsto x$, $y\mapsto -y$. On the other hand, suppose $\varphi$ is an automorphism which transforms $f_0$ to a normal form equation with parameter $\lambda_1+\lambda_2t$, and let $\kappa$ be the conjugation on $K$. Then $\kappa\circ\varphi\circ\kappa^{-1}$ is also an automorphism and transforms $f_0$ to the same normal form equation with the parameter replaced by $\kappa(\lambda_1+\lambda_2t)=\lambda_1-\lambda_2t$. Hence, $\lambda_1+\lambda_2t=-(\lambda_1-\lambda_2t)$ or $\lambda_1+\lambda_2t=\lambda_1-\lambda_2t$. Thus, $\lambda_1=0$ or $\lambda_2=0$, and the square of the parameter is a rational number. Therefore, $z^2-|b|^{-r}(de)^2$ is a polynomial with rational coefficients. 

Consider now the case where $r$ is even.  We either
have $a=+|b|^{\frac{r}{2}}de$ or $a=-|b|^{\frac{r}{2}}de$. By Theorem 33 of \cite{realclassify2} the parameter $a$ is uniquely determined.  Hence, by a similar argument as above, $|b|^{-\frac{r}{2}}a$ and thus $de$ is invariant under conjugation, and we conclude that $de\in\Q$.

The case where $r$ is even and $de<0$ is handled in lines \ref{line:caseIV2<0} - \ref{line:caseIV2<02}.
Taking the transformation $x\mapsto y$, $y\mapsto x$ into account, $f$ is both of type $Y_{r,r}^{\sigma+}$ with $a=|b|^{\frac{r}{2}}de$, and of type $Y_{r,r}^{\sigma-}$ with $a=-|b|^{\frac{r}{2}}de$. 

The case where $r$ is even and $de>0$ is handled in lines \ref{line:caseIV2m2} - \ref{line:caseIV2(b)}. 
Since $de>0$, the coefficient of $x^r$ and the parameter $a$ in the normal form equation have the same sign. For any polynomial $dx^r+bx^2y^2+ey^r$ in the right equivalence class of $f$, the sign of $d$ is the same. Since the extension $\mathbb Q\subset K=\mathbb Q[t]$ is quadratic, $d$ is of the form $d=\lambda_1+\lambda_2 t$, with $\lambda_i\in\mathbb Q$. Hence, for $\widetilde d(T)=\lambda_1+\lambda_2 T\in\mathbb Q[T]$ and $t_{1,2}=\pm \sqrt{D}$, we have $\sign(\widetilde d(t_1))=\sign(\widetilde d(t_2))$. Thus $\sign(d)=\sign(\widetilde d(0))=\sign(\lambda_1)$.

\end{itemize}

\subsection{The \texorpdfstring{$\boldsymbol{\widetilde Y_{r}}$}{\widetilde Y_r} case}%
\label{sec Yr}
In this case the general algorithm can be followed, except for a slight modification in Step (III). Note that  there are no monomials in $\supp(f)$ underneath $\Gamma(T)$ for $T=\widetilde Y_{r}$. We start by removing monomials in $f$ on $\Gamma(T)$ that are not in $\supp(T)$. To do this, we apply a linear transformation to $f$ such that $\supp(\jet(f,d))=\supp(T,d)$ for $d=4$. This transformation is implemented in Algorithm \ref{alg:YrLinear}.

\begin{algorithm}[h]
\caption{}%
\label{alg:YrLinear}
\begin{algorithmic}[1]
\Require{$f\in\Q[x,y]$ of real type $\widetilde Y_r$}
\Ensure{$h\in\Q[x,y]$ with $f\sim h$  and $\supp(\jet(h,4))=\supp(T,4)$}
\State $g:=\jet(f,4)$
\State factorize $g$ over $\Q$ as $cf_1^2$ with $f_1=x^2+axy+by^2$, and $a,b,c\in\Q$
\State apply $x\mapsto x+\frac{a}{2}y$, $y\mapsto y$ to $f$
\Return $f$
\end{algorithmic}
\end{algorithm}

Next we remove the monomials in $f$ above $\Gamma(T)$ not in $\supp(T)$. Note that the Arnold parameter system of germs of type $\widetilde Y_r$ contains only the monomial $x^r$. Hence, similarly to the case of germs that are equivalent to a germ with a non-degenerate Newton boundary, we iteratively consider the degree $j$ part of $f$, for increasing $j$ with $d_w<j<r$. Since the coefficients of the elements 
of a basis of $\mathbb R[x,y]/\langle \frac{\partial{f}}{\partial{x}},\frac{\partial{f}}{\partial{y}}\rangle$ in degree $j$ are zero in the Arnold normal form, we can write this polynomial as
\[\frac{\partial f_0}{\partial x}v_1+\frac{\partial f_0}{\partial y}v_2,\]
with $v_1,v_2\in\R[x,y]$. Applying the transformation $x\mapsto x-v_1$, $y\mapsto y-v_2$ results in transforming the part of $f$ of  degree $j$ to $0$. Similarly, we then can transform the degree $r$ part of $f$ to $cx^r$, $c\in \mathbb{R}$. Since a germ of type $\widetilde Y_r$ is $r$-determined, terms of higher degree can be discarted. We summarize this procedure in Algorithm~\ref{alg:Yrabove}.

\begin{algorithm}[h]
\caption{Step (III) in the general algorithm for the real main types $\widetilde{Y}_{r}$}%
\label{alg:Yrabove}
\begin{algorithmic}[1]
\Require{$f\in\Q[x,y]$ of real type $\widetilde Y_r$ such that $\supp(\jet(f,4))=\supp(T,4)$}
\Ensure{$h\in\Q[x,y]$ with $f\sim h$  and $\supp(\jet(h,r))=\supp(T,r)$
}
\State $g_0=\jet(f,4)=aq^2$, \quad$q=x^2+by^2\in\mathbb R[x,y]$, $a,b\in\mathbb R$
\State $i:=5$
\State $p_5 = (\jet(f,5)-\jet(f,4))/(4aq)$
\While{$i\neq r+1$}
  \State $px_i = p_i(x,0)$;
  \State $py_i = p_i-px_i$;
  \State apply $x\mapsto x-\frac{px_i}{x}$, $y\mapsto y- \frac{py_i}{by}$ to $f$;
  \State $i=i+1$
   \State $p_i = (\jet(f,i)-\jet(f,4))/(4aq)$
\EndWhile

\Return $f$
\end{algorithmic}
\end{algorithm}

\pagebreak[4]

\bibliographystyle{elsarticle-harv}

\end{document}